\documentclass[11pt]{article}
\usepackage[margin=1in]{geometry}
\usepackage{graphicx,xcolor,cite,mathtools}
\usepackage{amssymb,amsmath,amsthm,mathrsfs,paralist,bm,esint,setspace}
\RequirePackage[colorlinks,citecolor=blue,urlcolor=blue]{hyperref}

\newcommand{\R}{{\mathbb{R}}}
\newcommand{\E}{\mathrm{E}}

\renewcommand{\d}{\mathrm{d}}

\newcommand{\e}{\mathrm{e}}
\newcommand{\Var}{\text{\rm Var}}
\newcommand{\lip}{\text{\rm Lip}}

\DeclareMathOperator{\Cov}{\text{\rm Cov}}

\DeclarePairedDelimiter\floor{\lfloor}{\rfloor}

\def\MR#1{\href{http://www.ams.org/mathscinet-getitem?mr=#1}{MR#1}}

\title{Ergodicity, CLT and asymptotic maximum of the Airy$_1$ process}
\author{ Fei  Pu
	}
\date{}                                           

\begin{document}
\newtheorem{stat}{Statement}[section]
\newtheorem{proposition}[stat]{Proposition}
\newtheorem*{prop}{Proposition}
\newtheorem{corollary}[stat]{Corollary}
\newtheorem{theorem}[stat]{Theorem}
\newtheorem{lemma}[stat]{Lemma}
\theoremstyle{definition}
\newtheorem{definition}[stat]{Definition}
\newtheorem*{cremark}{Remark}
\newtheorem{remark}[stat]{Remark}
\newtheorem*{OP}{Open Problem}
\newtheorem{example}[stat]{Example}
\newtheorem{nota}[stat]{Notation}
\numberwithin{equation}{section}
\maketitle

\begin{abstract}
          We first show that the Airy$_1$ process is associated using the association property of the solution to the stochastic heat 
          equation and convergence of the KPZ equation to the KPZ fixed point.  Then we apply Newman's inequality to establish
          the ergodicity and central limit theorem for the Airy$_1$ process.  Combined with the asymptotic behavior of the tail probability, we derive a Poisson limit theorem for the Airy$_1$ process and give a precise estimate on the asymptotic behavior of the maximum of the Airy$_1$ process over an interval.  Analogous results for the Airy$_2$ process are also presented. 

\end{abstract}

\bigskip\bigskip



{
}

\section{Introduction and main results}\label{sec:int}

The Airy$_1$ process, $\{\mathcal{A}_1(x): x\in \R\}$, derived by Sasamoto \cite{Sas05} in the context of TASEP,  is a stationary stochastic process whose finite-dimensional distributions are given in terms of the Fredholm determinant (see \cite[Section 4.1.2]{WFS17}).  In particular, the one-point distribution is a scalar multiple of the GOE Tracy-Widom distribution, namely,
\begin{align*}
\mathrm{P}\{\mathcal{A}_1(0)\leq s\} =F_1(2s), \quad s\in\R,
\end{align*}
where $F_1$ denotes the GOE Tracy-Widom distribution.
 Moreover, it was proven by Quastel and Remenik \cite{QuR13} that the sample paths are H\"older continuous with exponent $\frac12-$. 

Recently, Basu et al. \cite{BBF23} have shown that the covariance of the Airy$_1$ process 
 decays super-exponentially by showing that $\Cov(\mathcal{A}_1(x)\,, \mathcal{A}_1(0))= \e^{-(\frac{4}{3}+o(1))x^3}$
as $x\to\infty$. This immediately implies a law of large numbers, i.e., as $N\to\infty$,
\begin{align*}
\frac1N\int_0^N \mathcal{A}_1(x)\ \mathrm{d} x \to \E[\mathcal{A}_1(0)], \quad \text{ in $L^2(\Omega)$}.
\end{align*}

We are aiming to establish the ergodicity and the matching central limit theorem for the Airy$_1$ process. 
\begin{theorem}\label{th:CLT}
          $\{\mathcal{A}_1(x): x\in \R\}$ is ergodic, and as $N\to\infty$,
          \begin{align}\label{CLT}
          \frac{1}{\sqrt{N}}\int_0^N (\mathcal{A}_1(x)-  \E[\mathcal{A}_1(x)])\ \mathrm{d} x \xrightarrow{\rm d} {\rm N}(0, \sigma^2), 
          \end{align}
          where $\sigma^2= \int_\R \Cov(\mathcal{A}_1(x)\,, \mathcal{A}_1(0)) \mathrm{d} x\in (0, \infty)$, and ``$\xrightarrow{\text{\rm d}\,}$'' denotes convergence
           in distribution.
\end{theorem}

Note that Quastel and Remenik \cite{QuR13} proved that the Airy$_1$ process fluctuates locally like a Brownian motion (see \cite[Theorem 3]{QuR13}), using a formula for the $n$-dimensional distributions of the Airy$_1$ process in terms of a Fredholm determinant on $L^2(\R)$.  
We refer to H\"agg \cite{Hag08} for the local central limit theorem for the Airy$_2$ process and Pimentel \cite{Pim18} for the local Brownian motion behavior of Airy processes.

The main tool to prove the ergodicity property and CLT \eqref{CLT} in Theorem \ref{th:CLT} is the association property of the Airy$_1$ process. 
We recall from Esary et al. \cite{EPW67} that
a random vector $X:=(X_1\,,\ldots,X_m)$ is said to be \emph{associated} if
	\begin{equation}\label{E:assoc}
		\Cov[h_1(X)\,,h_2(X)]\ge0,
	\end{equation}
for every pair of functions $h_1,h_2:\R^m\to\R$
that are nondecreasing in every coordinate and satisfy $h_1(X),h_2(X)\in L^2(\Omega)$. A random field $\Phi=\{\Phi(x)\}_{x\in\R^d}$ is
	\emph{associated} if $(\Phi(x_1)\,,\ldots,\Phi(x_m))$ is associated
	for every $x_1,\ldots,x_m\in\R^d$.
We remark that an associated random vector is also called to satisfy the FKG inequalities; see Newman \cite{New80}.

\begin{theorem}\label{th:association}
           $\{\mathcal{A}_1(x): x\in \R\}$ is associated. 
\end{theorem}

The association property of the Airy$_1$ process leads to a fundamental inequality, known as Newman's inequality (see \eqref{Newman} below), which plays a crucial role in proving ergodicity and CLT for the Airy$_1$ process. Furthermore, Newman's inequality implies an inequality for characteristic functions (see \eqref{Newman3} below), which
together with the asymptotic behavior of the GOE Tracy-Widom distribution,
 enables us to derive a Poisson limit theorem for high-level exceedances of the Airy$_1$ process.
 
Denote by $\rm T$ the inverse function of $s\mapsto 1-F_1(s)$. We have the following Poisson approximation for the Airy$_1$ process.

\begin{theorem}\label{th:Poisson}
          Fix $\lambda>0$. Let $\{x_{k, N}\}_{N\geq1, 1\leq k\leq N} \subset \R$ satisfy
          \begin{align}\label{distance}
          \liminf_{N\to\infty}\frac{\min_{1\leq j\neq k\leq N}|x_{j,N}-x_{k,N}|}{(3\log N)^{1/3}}>1. 
          \end{align}
          Then, as $N\to\infty$, 
          \begin{align}\label{eq:Poisson}
          \sum_{k=1}^N\bm{1}_{\{\mathcal{A}_1(x_{k,N})>\frac12 {\rm T}(\lambda/N)\}}\xrightarrow{\rm d} {\rm Poisson}(\lambda),
          \end{align}
          where ${\rm Poisson}(\lambda)$ denotes a Poisson random variable with parameter $\lambda$.
\end{theorem}

The function $\rm T$ (inverse of $s\mapsto 1-F_1(s)$) is introduced in \eqref{eq:Poisson}  so that the expectation of the sum in \eqref{eq:Poisson} is equal to 
the constant $\lambda$, as in the classical Poisson approximation theorem for independent Bernoulli random variables.
Theorem \ref{th:Poisson} states that when the points $x_{k, N}$, $k=1, \ldots, N$ are far apart (namely, satisfy the distance condition \eqref{distance}), then the identically distributed Bernoulli random variables given by the indicator functions in \eqref{eq:Poisson} are approximately independent and hence the sum of them will be approximated by the Poisson distribution.

Conus et al. \cite{CJK13} studied the spatial asymptotic behavior for the solution to the stochastic heat equation using a localization argument (see Chen \cite{Che16} for precise  spatial asymptotic behavior for the parabolic Anderson equation).
This localization argument is based on the mild formula for the stochastic heat equation and does not apply to the Airy$_1$ process. However, the association property provides a way to analyze the independence structure of the Airy$_1$ process in terms of the asymptotic behavior of the covariance (see the inequality \eqref{eq:prob} below). 
Combined with the asymptotic behavior of the tail probability of the maximum of the Airy$_1$ process (see Proposition \ref{prop:tail} below), we can study the precise asymptotic behavior of 
the maximum of the Airy$_1$ process.

\begin{theorem}\label{maximum}
          Almost surely,
          \begin{align}\label{eq:max}
          \lim_{N\to\infty} \frac{\max_{0\leq x\leq N}\mathcal{A}_1(x)}{(\log N)^{2/3}} = \frac12\left(\frac32\right)^{2/3}.
          \end{align}
\end{theorem}
As for the Airy$_2$ process $\{\mathcal{A}_2(x): x\in \R\}$, we have the following. 
\begin{theorem}\label{th:A2}
          Almost surely,
          \begin{align}\label{eq:A2}
          \left(\frac14\right)^{2/3}\leq \liminf_{N\to\infty} \frac{\max_{0\leq x\leq N}\mathcal{A}_2(x)}{(\log N)^{2/3}} \leq  \limsup_{N\to\infty} \frac{\max_{0\leq x\leq N}\mathcal{A}_2(x)}{(\log N)^{2/3}} \leq \left(\frac34\right)^{2/3}.
          \end{align}
\end{theorem}

The paper is organized as follows.  After introducing some preliminaries in  Section \ref{sec:pre}, we
first prove Theorem \ref{th:association} in Section \ref{sec:as}.  Then we establish the ergodicity and CLT in Section \ref{sec:ergodicityCLT}. The Poisson approximation (Theorem \ref{th:Poisson}) is established in Section \ref{Sec:Poisson}.
Finally, we prove Theorems \ref{maximum} and \ref{th:A2} in Section \ref{sec:max}.

Let us close the introduction with a brief description of the notation of this paper. For a Lipschitz continuous
function $g:\R^m\mapsto \R$,  define 
\[
	\lip(g) := \sup_{a, b\in \R^m, a\neq b}
	\frac{|g(b)-g(a)|}{\|b-a\|}
\]
where $\|\cdot\|$ denotes the Euclidean norm on $\R^m$, and for $j=1, \ldots, m$, we define
\begin{align*}
\lip_j(g):= \sup_{ a_j\neq b_j} \frac{|g(a)-g(b)|}{|a_j-b_j|},
\end{align*}
where $a_j$ and $b_j$ are the $j$-th coordinates of $a$  and $b$ respectively.

\section{Preliminaries}
\label{sec:pre}

In this section, we will first recall some basic facts on associated random variables. We refer to \cite{PR12} for more details on 
associated sequences. 

Recall the definition of association in \eqref{E:assoc}.	
Esary et al.
\cite{EPW67} showed that the random vector $X$ is associated if and only if
\eqref{E:assoc} holds for all bounded and continuous functions $h_1,h_2$ that are nondecreasing coordinatewise. Using approximation and dominated convergence theorem, in order to verify that $X$ is associated, it suffices to check that \eqref{E:assoc} holds for all  $h_1,h_2$ which are smooth, nondecreasing coordinatewise and have bounded partial derivatives. 

Another concept of dependence is positive quadrant dependence. According to  \cite{Leh66}, a pair of random variables $X$, $Y$ are said to be positive quadrant dependent if
\begin{align*}
\mathrm{P}\{X\leq x, Y\leq y\}- \mathrm{P}\{X\leq x\}\mathrm{P}\{Y\leq y\}\geq 0
\end{align*}
for all $x, y\in \R$. 
For an associated random vector $(X, Y)$, it is clear that the random variables $X$ and $Y$ are positive quadrant dependent. Therefore, according to Newman \cite[Lemma 3]{New80},  for an associated random vector $(X, Y)$ such that both $X$
and $Y$ have finite variance, we have
\begin{align}\label{Newman}
\left|\Cov(f(X)\,, g(Y))\right| \leq \|f'\|_{\infty}\|g'\|_{\infty}\Cov(X\,, Y),
\end{align}
for $C^1$ complex valued functions $f$, $g$ on $\R$ with $f'$, $g'$ bounded, where $\|\cdot\|_{\infty}$ denotes the sup norm on $\R$. Inequality \eqref{Newman} is known as Newman's inequality.  Bulinski \cite{Bul96} generalized \eqref{Newman} to 
Lipschitz continuous functions (see also \cite[Theorem 6.2.4]{PR12}). Moreover, according to Bulinski and Shabanovich \cite{BuS98} (see also \cite[Theorem 6.2.6]{PR12}),
letting $(Y_1, \ldots, Y_m)$ be an associated random vector such that each component has finite variance, then we have
for Lipschitz continuous functions $f, g: \R^m\mapsto \R$, 
\begin{align}\label{Newman2}
          \left|\Cov(f(Y_1, \ldots, Y_m)\,, g(Y_1, \ldots, Y_m))\right|\leq \sum_{j=1}^m\sum_{\ell=1}^m
           \lip_j(f)\lip_\ell(g) \Cov(Y_j\,, Y_\ell),
\end{align}
where $\lip_j(f), \lip_\ell(g)$ are defined at the end of Section \ref{sec:int}. Inequality \eqref{Newman2} generalizes \cite[(12)]{New80}, which is a
 multivariate version of \eqref{Newman}.

Another consequence of Newman's inequality \eqref{Newman} is the following inequality for the characteristic functions of associated random vector. Indeed, by \cite[Theorem 1]{New80}, for an associated random vector $(Y_1, \ldots, Y_m)$ such that each component has finite variance, we have for any real numbers $r_1, \ldots, r_m$,
\begin{align}\label{Newman3}
\bigg|\E\left[\e^{{\rm i}\sum_{j=1}^m r_jY_j}\right]
- \prod_{j=1}^m \E\left[\e^{{\rm i}r_jY_j}\right] 
\bigg| \leq \frac12\sum_{\substack{1\leq j, \ell\leq m\\j \neq \ell}} |r_j||r_\ell| \Cov(Y_j\,, Y_\ell).
\end{align}
The study of central limit theorem for associated random sequences essentially relies on the above inequality; see \cite[Theorem 2]{New80}. We will see in Section \ref{Sec:Poisson} that inequality \eqref{Newman3} also plays a crucial role in the study of Poisson approximation for the Airy$_1$ process. In order to prove Theorems \ref{maximum} and \ref{th:A2}, we need the following 
probability inequalities for associated random vectors. 

\begin{lemma}
          Let $(Y_1, \ldots, Y_m)$ be an associated random vector. Let $y_1, \ldots, y_m$ be real numbers. 
           For subsets $A, B$ of $\{1, \ldots, m\}$, 
          \begin{align}\label{FKG1}
          &\mathrm{P}\{Y_j\leq y_j, j\in A\cup B\} - \mathrm{P}\{Y_j\leq y_j, j\in A\} \mathrm{P}\{Y_k\leq y_k, k\in B\}\nonumber\\
          &\qquad \qquad \qquad \qquad\qquad\qquad \leq \sum_{j\in A}\sum_{k\in B}
           \left(\mathrm{P}\{Y_j\leq y_j, Y_k\leq y_k\} - \mathrm{P}\{Y_j\leq y_j\}\mathrm{P}\{Y_k\leq y_k\}\right).
          \end{align}
          As a consequence, 
          \begin{align}\label{FKG2}
          &\mathrm{P}\{Y_j\leq y_j, 1\leq j\leq m\} - \prod_{j=1}^m\mathrm{P}\{Y_j\leq y_j\}  \nonumber\\
            &\qquad \qquad \qquad \qquad\qquad \leq\sum_{1\leq j<k\leq m} 
          \left(\mathrm{P}\{Y_j\leq y_j, Y_k\leq y_k\} - \mathrm{P}\{Y_j\leq y_j\}\mathrm{P}\{Y_k\leq y_k\}\right).
          \end{align}
\end{lemma}
\begin{proof}
          The proof of \eqref{FKG1} is similar to that of \cite[Theorem 1.2.2]{PR12} (see also \cite[Lemma 1]{Leb72}). We define
          the random variables $Z_j= \bm{1}_{\{Y_j\leq y_j\}}$, for $j=1, \ldots, m$. It is clear that $(Z_1, \ldots, Z_m)$ 
          is associated. Then we can follow along the same lines as in the proof of \cite[Theorem 1.2.2]{PR12} to
           conclude \eqref{FKG1}.  
           
           Clearly, the inequality \eqref{FKG2} holds with $m=2$. Assume that the inequality \eqref{FKG2} is true for $m-1$. 
           Then for integer $m$,
           \begin{align*}
           &\mathrm{P}\{Y_j\leq y_j, 1\leq j\leq m\} - \prod_{j=1}^m\mathrm{P}\{Y_j\leq y_j\}\\
           &\qquad\qquad=            \mathrm{P}\{Y_j\leq y_j, 1\leq j\leq m\} -
           \mathrm{P}\{Y_j\leq y_j, 1\leq j\leq m-1\}\mathrm{P}\{Y_m\leq y_m\}\\
           &\qquad\qquad\quad + \mathrm{P}\{Y_j\leq y_j, 1\leq j\leq m-1\}\mathrm{P}\{Y_m\leq y_m\}
- \prod_{j=1}^m\mathrm{P}\{Y_j\leq y_j\}\\
&\qquad\qquad \leq \sum_{j=1}^{m-1}\left(\mathrm{P}\{Y_i\leq y_j, Y_m\leq y_m\}- \mathrm{P}\{Y_j\leq y_j\}\mathrm{P}\{Y_m\leq y_m\}\right)\\
&\qquad\qquad \quad + \sum_{1\leq j<k\leq m-1} \left(\mathrm{P}\{Y_j\leq y_j, Y_k\leq y_k\} - \mathrm{P}\{Y_j\leq y_j\}\mathrm{P}\{Y_k\leq y_k\}\right)\\
&\qquad\qquad= \sum_{1\leq j<k\leq m} \left(\mathrm{P}\{Y_j\leq y_j, Y_k\leq y_k\} - \mathrm{P}\{Y_j\leq y_j\}\mathrm{P}\{Y_k\leq y_k\}\right),
           \end{align*} 
           where in the inequality, we have used \eqref{FKG1} and the assumption for $m-1$. Hence, we prove the inequality 
           \eqref{FKG2} by induction.        
\end{proof}

We next introduce some facts on the  GOE and GUE Tracy-Widom distributions, denoted by $F_1$ and $F_2 $ respectively. Let $F_1'$ and $F_2'$ be the derivatives of $F_1$ and $F_2$ respectively, which correspond to the probability density functions of the GOE and GUE Tracy-Widom distributions. Recall from \cite{BBD08} (see also \cite{TW94, TW96}) that $F_1$ and $F_2$ can be written as
\begin{align}\label{TW1}
F_1(s)=F(s)E(s),\qquad F_2(s)=F(s)^2,
\end{align}
where
\begin{align}\label{TW2}
F(s)=\exp\left(-\frac12\int_s^{\infty}R(r)\mathrm{d} r\right), \quad E(s)=\exp\left(-\frac12\int_{s}^{\infty}q(r)\mathrm{d} r\right).
\end{align}
Here the (real) function $q$ is the solution to the Painlev\'e II equation
\begin{align*}
q''=2q^3+sq,
\end{align*}
that satisfies the boundary condition
\begin{align}\label{TW3}
q(s)\sim \frac{1}{2\sqrt{\pi}s^{1/4}}\e^{-\frac23s^{3/2}}, \quad s\to+\infty.
\end{align}
The function $R$ is defined as
$
R(s)=\int_s^\infty(q(r))^2\mathrm{d} r. 
$
Taking derivative in \eqref{TW1} and using the formulas in \eqref{TW2},  we have
\begin{align}\label{TW4}
F_1'(s)= \frac12F_1(s)\left(R(s)+q(s)\right).
\end{align}
According to the asymptotics of the functions $q, R$ and $F_1$ as $s\to -\infty$ (see \cite[(11), (12), (18)]{BBD08}), we derive from \eqref{TW4} that the probability density function $F'_1$ is continuous and bounded. 
Similarly, we have
$
F_2'(s)=F_2(s)R(s),
$
and use the  asymptotics of the functions $R$ and $F_2$ as $s\to -\infty$ (see \cite[(12), (16)]{BBD08}), we deduce that 
the GUE Tracy-Widom distribution also has a bounded and continuous probability density function. 

Moreover, in light of \eqref{TW3}, we deduce that for all sufficiently large $s$,
\begin{align*}
R(s)\leq \int_s^{\infty} \frac{1}{\sqrt{\pi}r^{1/4}}\e^{-\frac23r^{3/2}} q(r)\mathrm{d} r \leq \e^{-\frac23s^{3/2}}.
\end{align*}
Hence, we combine with \eqref{TW4} and \eqref{TW3} to obtain that there exists a constant $C_1>0$ such that
\begin{align}\label{TW5}
F_1'(s) \leq \e^{-\frac23s^{3/2}}, \quad \text{for all $s\geq C_1$}.
\end{align}
Furthermore,  we see from \cite[Theorem 1]{DuV13} that for fixed $\epsilon\in (0, \frac23)$, there exists a positive constant $C_2$ depending on $\epsilon$ such that
\begin{align} \label{TWtail}
\e^{-(\frac23+\epsilon)s^{3/2}}\leq 1-F_1(s) \leq \e^{-(\frac23-\epsilon)s^{3/2}}, \quad \text{for all $s\geq C_2$}.
\end{align}
See also \cite[(25), (26)]{BBD08} for the precise upper tail of GOE Tracy-Widom distribution.  

Recall that $\rm T$ denotes the inverse function of $s\mapsto 1-F_1(s)$. The following lemma will be used when we prove Theorem \ref{th:Poisson} in Section \ref{Sec:Poisson}.

\begin{lemma}\label{lem:diff}
          Fix $\lambda>0$ and $\epsilon\in (0, \frac23)$. There exist an integer $N_0$ and a positive constant $C$
           both depending on 
          $\lambda$ and $\epsilon$ such that for all $z\geq N_0$
          \begin{align}\label{diff}
          {\rm T}(\lambda/z)-{\rm T}(\lambda/(z-1)) \geq Cz^{-\frac{2+6\epsilon}{2+3\epsilon}}.
          \end{align}
\end{lemma}
\begin{proof}
          Taking derivative on both sides of the following identity
          \begin{align*}
          \int_{{\rm T}(\lambda/z)}^{\infty}F'_1(s)\mathrm{d} s =\frac{\lambda}{z}, \quad z>\lambda,
          \end{align*}
          we observe that for all $z>\lambda$
          \begin{align}\label{density}
          \frac{\rm d}{\mathrm{d} z}{\rm T}(\lambda/z)=\frac{\lambda}{z^2F_1'({\rm T}(\lambda/z))}.
          \end{align}
          We can choose a sufficiently large integer $N_0$ depending on $\lambda>0$ and $\epsilon\in (0, \frac23)$ such that 
          \begin{align*}
          {\rm T}(\lambda/z) > C_1\vee C_2, \quad \text{for all $z\geq N_0-1$},
          \end{align*}
          where $C_1, C_2$ are the constants in \eqref{TW5} and \eqref{TWtail} respectively.  Hence, 
          by \eqref{density} and \eqref{TW5}, we derive that for all $z\geq N_0-1$,
          \begin{align*}
          \frac{\rm d}{\mathrm{d} z}{\rm T}(\lambda/z)\geq \frac{\lambda}{z^2}\e^{\frac23([{\rm T}(\lambda/z)]^{3/2})}
          \geq  \frac{\lambda}{z^2}\left(\frac{z}{\lambda}\right)^{\frac{\frac23}{\frac23+\epsilon}} 
          =Cz^{-\frac{2+6\epsilon}{2+3\epsilon}},
          \end{align*}
          where the second inequality holds by the first inequality in \eqref{TWtail}. 
          Finally, the above estimate and the mean-value theorem imply \eqref{diff}.       
\end{proof}

\section{Association}\label{sec:as}

We prove the association property of the Airy$_1$ process in this section.  We first consider the stochastic heat equation
\begin{align}\label{SHE}
\partial_tZ(t\,,x) = \delta \partial_x^2Z(t\,,x) + \frac14\delta^{1/2}Z(t\,,x) \xi(t\,,x), \quad t>0, x\in\R
\end{align}
subject to flat initial data $Z(0)\equiv 1$ or narrow wedge initial data $Z(0)=\delta_0$. In the above equation,
$\delta$ is a positive constant and $\xi$ denotes space-time white noise. 
Chen et al. \cite[Theorem A.4]{CKNP23} have established the association property for stochastic heat equation with general diffusion coefficient and initial data, by showing that the Malliavin derivative of the solution is non-negative almost surely (see \cite[(A.4)]{CKNP23}) and then applying the Clark-Ocone formula and Ito's isometry to show that the covariance 
\begin{align*}
\Cov(h_1(Z(t_1\,, x_1), \ldots, Z(t_m\,, x_m))\,, h_2(Z(t_1\,, x_1), \ldots, Z(t_m\,, x_m)))
\end{align*}
is nonnegative for all  $h_1,h_2$ which are smooth, nondecreasing coordinatewise and have bounded partial derivatives. 
Although the constants (depending on $\delta$) in equation \eqref{SHE} are different from those in equation (1.1) of \cite{CKNP23}, Theorem A.4 of 
\cite{CKNP23} ensures that the solution to \eqref{SHE} is associated.

\begin{theorem}[{{\cite[Theorem A.4]{CKNP23}}}]\label{th:SHE:ass}
          For every fixed $\delta>0$, the solution $\{Z(t\,,x): (t, x)\in \R_+\times \R\}$ to \eqref{SHE} 
          with flat or narrow wedge initial condition is associated. 
\end{theorem}

Now we prove Theorem \ref{th:association}.
\begin{proof}[{Proof of Theorem \ref{th:association}}]
          For $\delta>0$, define
          \begin{align}\label{KPZ}
          \tilde{h}_\delta(t\,, x)= 4\delta\log Z(t,x) + \frac{t}{12}, \quad (t,x)\in (0, \infty)\times\R,
          \end{align}
          where $Z(t\,,x)$ is the solution to \eqref{SHE} subject to flat initial condition $Z(0)\equiv1$. It is clear that for each $\delta>0$, the process 
          $\tilde{h}_\delta=\{\tilde{h}_\delta(t,x): (t,x)\in (0, \infty)\times \R\}$ is associated. Let  $\mathfrak{h}=\{\mathfrak{h}(t, x; 0):
           (t, x)\in (0, \infty)\times \R\}$ be the KPZ fixed point (see \cite{MQR21}) with flat initial data. 
            According to \cite[Theorem 2.2(3)]{QuS23}, 
           as $\delta\to0$, the finite-dimensional distributions of $\tilde{h}_\delta$ converge to those of $\mathfrak{h}$; 
           see also \cite[Theorem 1]{Vir20}.
           Since convergence in distribution preserves the association property (see (P$_5$) of \cite{EPW67}), we see that $\mathfrak{h}$ is associated. 
           Moreover, because $\{\mathfrak{h}(1, x; 0): x\in\R\}$ has the same finite-dimensional distributions as $\{2^{1/3}\mathcal{A}_1(2^{-2/3}x): x\in\R\}$
            (see \cite[(4.15)]{MQR21}), we conclude that the Airy$_1$ process is associated. 
\end{proof}

\begin{proposition}\label{A2:association}
          The Airy$_2$ process $\{\mathcal{A}_2(x): x\in \R\}$ is associated.
\end{proposition}
\begin{proof}
          The proof is similar to that of Theorem \ref{th:association}. Denote by   $\{\mathfrak{h}(t, x; \partial_0):
           (t, x)\in (0, \infty)\times \R\}$ the KPZ fixed point starting from narrow wedge at $0$.  Let $\tilde{h}_\delta(t\,, x)$
           be defined as in \eqref{KPZ}, where $Z(t\,,x)$ is the solution to \eqref{SHE} subject to narrow wedge initial condition $Z(0)\equiv\delta_0$. By Theorem \ref{th:SHE:ass}, we know that for each $\delta>0$, the process $\{\tilde{h}_\delta(t,x): (t,x)\in (0, \infty)\times \R\}$ is associated.  Theorem 2.2(3) of \cite{QuS23} ensures that 
           the finite-dimensional distributions of $\{\tilde{h}_\delta(t,x): (t,x)\in (0, \infty)\times \R\}$ converge to those of 
           $\{\mathfrak{h}(t, x; \partial_0):
           (t, x)\in (0, \infty)\times \R\}$ as $\delta\to0$. This implies that  the process $\{\mathfrak{h}(t, x; \partial_0):
           (t, x)\in (0, \infty)\times \R\}$ is associated. In particular, the process $\{\mathfrak{h}(1, x; \partial_0):
           x\in \R\}$ is associated and hence the process $\{\mathfrak{h}(1, x; \partial_0)+x^2:
           x\in \R\}$ is also associated. Moreover, since the process $\{\mathfrak{h}(1, x; \partial_0)+x^2:
           x\in \R\}$ has the same finite-dimensional distributions as the Airy$_2$ process $\{\mathcal{A}_2(x): x\in\R\}$
           (see \cite[(4.14)]{MQR21}), we conclude that the Airy$_2$ process is associated.
\end{proof}

\begin{remark}
          The association property of Airy processes can also be derived from last passage percolation by extending the argument in \cite[Lemma 3.2]{BBF23}; see also \cite[Remark 2]{BB24}.
\end{remark}

\section{Ergodicity and CLT}\label{sec:ergodicityCLT}

We prove the ergodicity and central limit theorem for the Airy$_1$ process in this section. We first prove the ergodicity. 

\begin{proof}[{Proof of  Theorem \ref{th:CLT}: ergodicity}]
          Chen et al. \cite[Lemma 7.2]{CKNP21} present a general criterion on the ergodicity for stationary process, 
          which has been modified to a variant by 
          Balan and Zheng \cite{BaZ23}. According to Balan and Zheng \cite[Lemma 4.2]{BaZ23}, 
          in order to prove that the Airy$_1$ process  is ergodic, we need to verify that 
          \begin{align}\label{variance}
          \lim_{N\to\infty}\frac{1}{N^2}\Var\bigg(\int_{[0, N]} G\Big(\sum_{j=1}^kb_i\mathcal{A}_1(x+\zeta_j)\Big)\mathrm{d} x\bigg)=0,
          \end{align}
          for all integers $k\geq 1$, for every $b_1, \ldots, b_k$, $\zeta_1, \ldots, \zeta_k\in \R$, and for $G\in \{x\mapsto \cos x, x\mapsto \sin x\}$.
          
          Notice that 
          \begin{align}\label{var2}
          &\frac{1}{N^2}\Var\bigg(\int_{[0, N]} G\Big(\sum_{j=1}^kb_i\mathcal{A}_1(x+\zeta_j)\Big)\mathrm{d} x\bigg) \nonumber\\
          & \qquad \qquad= \frac{1}{N^2}\int_{[0, N]^2}\mathrm{d} x\mathrm{d} y\, \Cov\bigg(G\Big(\sum_{j=1}^kb_j\mathcal{A}_1(x+\zeta_j)\Big)\,, 
          G\Big(\sum_{j=1}^kb_j\mathcal{A}_1(y+\zeta_j)\Big)\bigg).
          \end{align}
          Since the Airy$_1$ process is associated, we apply the inequality \eqref{Newman2} with $Y_j= \mathcal{A}_1(x+\zeta_j)$ for $j=1, \ldots, k$, $Y_j= \mathcal{A}_1(y+\zeta_{j-k})$ for
           $j=k+1, \ldots, 2k$,
          $f(y_1, \ldots, y_{2k})=G(\sum_{j=1}^kb_jy_j)$, and $g(y_1, \ldots, y_{2k})=G(\sum_{j=k+1}^{2k}b_{j-k}y_{j})$, to see that
          \begin{align*}
           & \Cov\bigg(G\Big(\sum_{j=1}^kb_i\mathcal{A}_1(x+\zeta_j)\Big)\,, 
          G\Big(\sum_{j=1}^kb_i\mathcal{A}_1(y+\zeta_j)\Big)\bigg)
          \leq  \sum_{j=1}^{k}\sum_{\ell=1}^k|b_j||b_\ell| \Cov(\mathcal{A}_1(x+\zeta_j)\,, \mathcal{A}_1(y+\zeta_\ell)).
          \end{align*}
          Here, we also use the fact that  $\lip_j(f)\leq |b_j|$ for $1\leq j\leq k$, $\lip_j(f)=0$ for $j=k+1, \ldots, 2k$ and 
          $\lip_\ell(g)=0$ for $1\leq \ell\leq k$, $\lip_\ell(g)\leq |b_{\ell-k}|$ for $\ell=k+1, \ldots, 2k$.
          According to \cite[Corollary 2.2]{BBF23}, there exist constants $c, c_1, c_2>0$ such that for all $x\in \R$,
          \begin{align*}
          \Cov(\mathcal{A}_1(x)\,, \mathcal{A}_1(0)) \leq \e^{cx^2}\e^{-\frac43|x|^{3}}\leq K(x):=c_1\e^{-c_2x^2}.
          \end{align*}
          Hence, the left-hand side of \eqref{var2} is bounded above by
          \begin{align*}
          &\frac{ \sup_{1\leq j\leq k}b_j^2}{N^2} \sum_{j=1}^{k}\sum_{\ell=1}^k\int_{[0, N]^2}K(x-y+\zeta_j-\zeta_\ell)\mathrm{d} x\mathrm{d} y = \left[ \sup_{1\leq j\leq k}b_j^2\right]
          \sum_{j=1}^{k}\sum_{\ell=1}^k \left(\bm{1}_N*\tilde{\bm{1}}_N*K\right)(\zeta_j-\zeta_\ell),
          \end{align*}
          where the function $\bm{1}_N$ is defined as $\bm{1}_N(x)= \frac1N\bm{1}_{[0, N]}(x)$ for $x\in \R$ and  $\tilde{\bm{1}}_N(x)=  \bm{1}_N(-x)$. 
          The function $\bm{1}_N*\tilde{\bm{1}}_N*K$ is non-negative definite and hence maximized at the origin. Hence we have
          \begin{align*}
          \frac{1}{N^2}\Var\bigg(\int_{[0, N]} g\Big(\sum_{j=1}^kb_i\mathcal{A}_1(x+\zeta_j)\Big)\mathrm{d} x\bigg)&\le
          \sup_{1\leq j\leq k}b_j^2 \, k^2 \left(\bm{1}_N*\tilde{\bm{1}}_N*K\right)(0)
           \leq  \sup_{1\leq j\leq k}b_j^2 \frac{k^2}{N}\int_{-N}^N K(x)\mathrm{d} x
          \end{align*}
          which converges to $0$ as $N\to\infty$. Thus, \eqref{variance} is verified and we complete the proof of ergodicity. 
\end{proof}

\begin{remark}
          We refer to \cite{CoS14, PrS02} for the ergodicity of the Airy$_2$ process.
\end{remark}

We proceed to prove the central limit theorem \eqref{CLT}.

\begin{proof}[{Proof of  Theorem \ref{th:CLT}: CLT}]
          We will apply \cite[Theorem 2]{New80} (see also \cite[Theorem 1.2.19]{PR12}) to prove the central limit theorem in Theorem \ref{th:CLT}. 
          We can assume that the limit in \eqref{CLT} is taken along integers. 
          Denote 
          \begin{align*}
          X_k= \int_{k}^{k+1}\mathcal{A}_1(x) \mathrm{d} x, \quad k\in\mathbb{Z}.
          \end{align*}
          Since the Airy$_1$ process is stationary and $\mathcal{A}_1(0)$ has finite second moment, the process $\{X_k: k\in\mathbb{Z}\}$ 
          is stationary and $X_0$ has 
          finite second moment.  For each $k\in\mathbb{Z}$, let 
          \begin{align*}
          X_k^{(n)}= \frac{1}{n}\sum_{j=1}^{n}\mathcal{A}_1(k+\frac{j}{n}), \quad n\geq 1.
          \end{align*}
          Clearly, for each $n\geq1$, the process $\{X_k^{(n)}: k\in \mathbb{Z}\}$ is associated. 
          Moreover, since the Airy$_1$ process has continuous sample paths (see \cite{QuR13}), we know that as $n\to\infty$, 
          the finite-dimensional distributions of 
          $\{X_k^{(n)}: k\in \mathbb{Z}\}$ converge to those of $\{X_k: k\in\mathbb{Z}\}$. Hence, the process $\{X_k: k\in\mathbb{Z}\}$ is associated.

          In order to apply \cite[Theorem 2]{New80}, it remains to verify that
          \begin{align*}
          \sigma^2:=\sum_{k\in\mathbb{Z}}\Cov(X_0\,, X_k)
          \end{align*}
          is finite. Indeed, by the stationarity of the Airy$_1$ process,
          \begin{align*}
          \sigma^2&=\sum_{k\in\mathbb{Z}}\Cov\left(\int_0^1\mathcal{A}_1(x)\mathrm{d} x\,, \int_k^{k+1}\mathcal{A}_1(y)\mathrm{d} y\right)
          =\sum_{k\in\mathbb{Z}} \int_0^1\mathrm{d} x\int_k^{k+1}\mathrm{d} y \Cov(\mathcal{A}_1(x)\,, \mathcal{A}_1(y))\\
          &= \int_0^1\mathrm{d} x\int_\R\mathrm{d} y \Cov(\mathcal{A}_1(x-y)\,, \mathcal{A}_1(0))= \int_\R \Cov(\mathcal{A}_1(x)\,, \mathcal{A}_1(0))\mathrm{d} x,
          \end{align*}
          which is finite by  \cite[Corollary 2.2]{BBF23}.
          
          Therefore, all the conditions in \cite[Theorem 2]{New80} are met and we conclude that as $N\to\infty$,
          \begin{align*}
          \frac{X_0+\ldots +X_{N-1}-N\E[X_0]}{\sqrt{N}}\xrightarrow{\rm d}{\rm N}(0, \sigma^2),
          \end{align*}
          which implies \eqref{CLT}.
\end{proof}

\begin{remark}
           (1) We can also apply \cite[Theorem 2]{New80} to derive a discrete version of \eqref{CLT}, that is, as $N\to\infty$,
          \begin{align*}
          \frac{\mathcal{A}_1(1)+\ldots+\mathcal{A}_1(N)-N\E[\mathcal{A}_1(0)]}{\sqrt{N}}\xrightarrow{\rm d}{\rm N}(0, \tau^2)
          \end{align*}
          with $\tau^2=\sum_{k\in\mathbb{Z}}\Cov(\mathcal{A}_1(k)\,, \mathcal{A}_1(0))$.
          
          (2) One can apply the Berry-Esseen theorem for associated sequences (see for instance \cite[p.15-16]{PR12}) to obtain the rate
          of convergence in the central limit theorem for the Airy$_1$ process. 
          
          (3) Recall that $\{\mathfrak{h}(t, x; 0):
           (t, x)\in (0, \infty)\times \R\}$ denotes the KPZ fixed point (see \cite{MQR21}) with flat initial data. 
           Using the scaling invariance property of the KPZ fixed point (see \cite[Theorem 4.5(i)]{MQR21}) and \eqref{CLT},
           we deduce that for fixed $t>0$, as $N\to\infty$,  
           \begin{align*}
           \frac{1}{\sqrt{N}}\int_0^N\left( \mathfrak{h}(t, x; 0)- \E[\mathfrak{h}(t, x; 0)]\right)\, \mathrm{d} x \xrightarrow{\rm d} 
           {\rm N}(0, 2^{4/3}t^{4/3}\sigma^2),
           \end{align*}
           where $\sigma^2$ is the quantity given in Theorem \ref{th:CLT}. One may expect that a corresponding 
           functional central limit theorem in $t$ also holds. 
\end{remark}

Similarly, we can also establish the central limit theorem for the Airy$_2$ process.
\begin{proposition}
           As $N\to\infty$,
          \begin{align*}
          \frac{1}{\sqrt{N}}\int_0^N (\mathcal{A}_2(x)-  \E[\mathcal{A}_2(x)])\ \mathrm{d} x \xrightarrow{\rm d} {\rm N}(0, \tilde{\sigma}^2), 
          \end{align*}
          where $\tilde{\sigma}^2= \int_\R \Cov(\mathcal{A}_2(x)\,, \mathcal{A}_2(0)) \mathrm{d} x\in (0, \infty)$.
\end{proposition}
\begin{proof}
          The proof follows along the same lines as that of \eqref{CLT}, applying \cite[Theorem 2]{New80} with the stationarity (see \cite{PrS02}), association (see Proposition \ref{A2:association}) of the Airy$_2$ process and the fact that the covariance of the Airy$_2$ process is integrable since the decay rate of  
          $\Cov(\mathcal{A}_2(x)\,, \mathcal{A}_2(0))$ is $x^{-2}$ as $x\to\infty$ (see \cite{Wid04}).
\end{proof}

We conclude this section by another application of Newman's inequality, which is used to study the Lebesgue measure of the set $x\in [0, N]$ where $\mathcal{A}_1(x)$ exceeds a high level.  The following result is analogous to \cite[Corollary 9.5]{CKNP21}, where the Poincar\'e inequality is used to estimate the covariance of the solution to stochastic partial differential equation. In the context of the Airy$_1$ process, when it comes to the estimate of covariance, Newman's inequality plays the role as the Poincar\'e inequality.

\begin{corollary}
          For $\alpha\in (0, 3/4)$, as $N\to\infty$,
          \begin{align}\label{con:pro}
          \frac{1}{\log N}\log\left(\int_0^N \bm{1}_{\{\mathcal{A}_1(x)>\frac12 (\alpha \log N)^{2/3}\}}\mathrm{d} x\right) \to  1-\frac{2}{3}\alpha, \quad \text{in probability}.
          \end{align}
\end{corollary}
\begin{proof}
          Choose and fix $\alpha\in (0, 3/4)$. For $N\geq1$, we define $a_N=\frac12(\alpha \log N)^{2/3}$ and the two functions
          \begin{align*}
          G_N(z):= 1\wedge(z-a_N+1)_+ \quad \text{and} \quad g_N(z)= 1\wedge (z-a_N)_+, \quad z\in\R. 
          \end{align*}
          Recall that $\mathrm{P}\{\mathcal{A}_1(0)\le s\}= F_1(2s)$, where $F_1$ denotes the GOE Tracy-Widom distribution function. 
          According to \cite[Theorem 1]{DuV13}, $s^{-3/2}\log\mathrm{P}\{\mathcal{A}_1(0)>\frac12 s\}\to -\frac{2}3$ as $s\to+\infty$. By stationarity of the Airy$_1$ process,
          we see that as $N\to\infty$, 
          \begin{align}\label{tail}
          \E \int_0^Ng_N(\mathcal{A}_1(x))\mathrm{d} x =N^{1-\frac{2}{3}\alpha +o(1)} \quad \text{and} \quad 
          \E \int_0^NG_N(\mathcal{A}_1(x))\mathrm{d} x =N^{1-\frac{2}{3}\alpha +o(1)}.
          \end{align}
          Using Chebychev's inequality, we have for any fixed $\epsilon \in (0, 1)$, 
          \begin{align}\label{cheby}
          &\mathrm{P}\left\{\left| \int_0^Ng_N(\mathcal{A}_1(x))\mathrm{d} x - \E  \int_0^Ng_N(\mathcal{A}_1(x))\mathrm{d} x \right| >\epsilon\,  \E  \int_0^Ng_N(\mathcal{A}_1(x))\mathrm{d} x \right\}\nonumber\\
          &\qquad\qquad\qquad\qquad\qquad\qquad\qquad
          \leq \Var\left( \int_0^Ng_N(\mathcal{A}_1(x))\mathrm{d} x\right)\epsilon^{-2}\left|\E  \int_0^Ng_N(\mathcal{A}_1(x))\mathrm{d} x\right|^{-2}.
          \end{align}
          Notice that for all $N\geq1$
          \begin{align}\label{var}
          \Var\left( \int_0^Ng_N(\mathcal{A}_1(x))\mathrm{d} x\right) &=\int_{[0, N]^2}\mathrm{d} x\mathrm{d} y\, \Cov(g_N(\mathcal{A}_1(x))\,, g_N(\mathcal{A}_1(y)))\nonumber\\
          &\leq [\lip(g_N)]^2 \int_{[0, N]^2}\mathrm{d} x\mathrm{d} y\, \Cov(\mathcal{A}_1(x)\,, \mathcal{A}_1(y))\nonumber\\
          &=  \int_{[0, N]^2}\mathrm{d} x\mathrm{d} y\, \Cov(\mathcal{A}_1(x)\,, \mathcal{A}_1(y)) \leq N\int_\R
           \Cov(\mathcal{A}_1(x)\,, \mathcal{A}_1(0)) \mathrm{d} x,
          \end{align}
          where the first inequality holds by the association property of the Airy$_1$ process and \eqref{Newman2}, and the second equality is because $\lip(g_N)=1$.
           Hence, we substitute \eqref{var}  into \eqref{cheby} to see that 
                  \begin{align*}
          &\mathrm{P}\left\{\left| \int_0^Ng_N(\mathcal{A}_1(x))\mathrm{d} x - \E  \int_0^Ng_N(\mathcal{A}_1(x))\mathrm{d} x \right| >\epsilon\,  
          \E  \int_0^Ng_N(\mathcal{A}_1(x))\mathrm{d} x \right\} \\
          &\qquad\qquad\qquad\qquad\qquad\qquad\qquad\qquad \qquad
           \leq \epsilon^{-2}N^{\frac{4}{3}\alpha -1+o(1)}\int_\R\Cov(\mathcal{A}_1(x)\,, \mathcal{A}_1(0)) \mathrm{d} x .
          \end{align*}
          Since $\alpha\in (0, 3/4)$, we obtain that as $N\to\infty$,
          \begin{align*}
          \frac{ \int_0^Ng_N(\mathcal{A}_1(x))\mathrm{d} x }{\E \int_0^Ng_N(\mathcal{A}_1(x))\mathrm{d} x } \to 1 \quad \text{in probability}.
          \end{align*}
          Taking logarithm and using \eqref{tail}, we conclude that as $N\to\infty$,
          \begin{align*}
          \frac{\log \left(\int_0^Ng_N(\mathcal{A}_1(x))\mathrm{d} x\right)}{\log N}\to 1-\frac{2}{3}\alpha  \quad \text{in probability}.
          \end{align*}
          Similarly, we can prove that 
          as $N\to\infty$,
          \begin{align*}
          \frac{\log \left(\int_0^NG_N(\mathcal{A}_1(x))\mathrm{d} x\right)}{\log N}\to 1-\frac{2}{3}\alpha  \quad \text{in probability}.
          \end{align*}
          Since $g_N\leq \bm{1}_{(a_N, \infty)} \leq G_N$, the preceding displays imply \eqref{con:pro} and hence complete the proof.
\end{proof}
\begin{remark}
          The convergence in \eqref{con:pro} also holds almost surely. This follows from a standard subsequencing 
          and sandwich  type argument (see the proof of \cite[Theorem 2.3.9]{Dur19}). We omit the details here.
\end{remark}

\section{Poisson approximation}\label{Sec:Poisson}

We prove Theorem \ref{th:Poisson} in this section.  The proof of Theorem \ref{th:Poisson} is based on inequality \eqref{Newman3} and some standard arguments; see \cite[Theorem 11]{New84} for a general result on limit theorems for sums of associated variables.

The indicator function in \eqref{eq:Poisson} is not Lipschitz continuous. In order to apply the inequalities \eqref{Newman2} and
\eqref{Newman3} to derive the Poisson limit theorem, we will first approximate the indicator function  in \eqref{eq:Poisson} by a Lipschitz continuous function. Fix $\lambda>0$. For $N\geq 2$, we introduce the function
\begin{align}\label{g_N}
g_N(z)= 1\wedge \frac{(z-\frac12{\rm T}(\lambda/N))_+}{[{\rm T}(\lambda/(N+1))-{\rm T}(\lambda/N)]/2}, \quad z\in \R,
\end{align}
where we recall that $\rm T$ denotes by the inverse function of $s\mapsto 1-F_1(s)$.
It is clear that for all $N\geq 2$,
\begin{align}\label{lowerupper}
g_N(z) \leq \bm{1}_{(\frac12{\rm T}(\lambda/N), \infty)}(z) \leq g_{N-1}(z), \quad \text{for all $z\in\R$},
\end{align}
and
\begin{align}\label{lip}
\lip(g_N) = \frac{2}{{\rm T}(\lambda/(N+1))-{\rm T}(\lambda/N)}.
\end{align}

We first prove the following result on Poisson approximation for the Airy$_1$ process. 
\begin{proposition}\label{prop:Poisson}
          Fix $\lambda>0$.
          Let $g_N$ be defined as in \eqref{g_N} for $N\geq 2$. Let $\{x_{k, N}\}_{N\geq1, 1\leq k\leq N}\subset \R$ satisfy the 
          condition in \eqref{distance}. Then, as $N\to\infty$,
          \begin{align}\label{Poisson:g}
          \sum_{k=1}^Ng_N(\mathcal{A}_1(x_{k,N})) \xrightarrow{\rm d} {\rm Poisson}(\lambda).
          \end{align} 
\end{proposition}
\begin{proof}
          {\bf Step 1}.
          Let $U_{1,N}, \ldots, U_{N, N}$ be independent random variables, each having the same distribution
           as $g_N(\mathcal{A}_1(0))$ (for this, we can enlarge the probability space if necessary). The first task will be to show 
           that
           \begin{align}\label{Poisson:U}
           \sum_{k=1}^NU_{k,N} \xrightarrow{\rm d} {\rm Poisson}(\lambda), \quad \text{as $N\to\infty$}.
           \end{align}

           Notice that the law of the family of the random variables $\sum_{k=1}^NU_{k,N}, N\geq1$ is tight. This is because that 
           by the definition of $g_N$ in \eqref{g_N},
           \begin{align*}
           \sup_{N\geq1}\E\left[\sum_{k=1}^NU_{k,N}\right] = \sup_{N\geq1}\left(N\E[g_N(\mathcal{A}_1(0))]\right) \leq 
            \sup_{N\geq1}\left(N\E\left[\bm{1}_{(\frac12{\rm T}(\lambda/N), \infty)}(\mathcal{A}_1(0))\right]\right)= \lambda,
           \end{align*}
           where the first inequality is due to the first inequality in \eqref{lowerupper}. Thus, in order to prove \eqref{Poisson:U}, 
           it suffices to prove that the moment generating function of $\sum_{k=1}^NU_{k,N}$ converges to that of ${\rm Poisson}
           (\lambda)$, as $N\to\infty$.
           
           For $\theta\in \R$, define
           \begin{align*}
           \phi_N(\theta)&:= \E\left[\e^{-\theta g_N(\mathcal{A}_1(0))}\right]\quad \text{and} \quad
           \Phi_N(\theta):=\E\left[\e^{-\theta \sum_{k=1}^NU_{k,N}}\right]=\left(\phi_N(\theta)\right)^N.
           \end{align*}
           Thanks to \eqref{lowerupper}, for all $\theta>0$, we have
           \begin{align*}
           \phi_N(\theta)\geq \E\left[\e^{-\theta \bm{1}_{(\frac12{\rm T}(\lambda/N), \infty)}(\mathcal{A}_1(0))}\right] 
           = \frac{\lambda}{N}\e^{-\theta} + (1-\frac{\lambda}{N}) = 1- \frac{\lambda}{N}(1-\e^{-\theta}),
           \end{align*}
           and 
           \begin{align*}
           \phi_N(\theta)\leq \E\left[\e^{-\theta \bm{1}_{(\frac12{\rm T}(\lambda/(N+1)), \infty)}(\mathcal{A}_1(0))}\right]
           =1- \frac{\lambda}{N+1}(1-\e^{-\theta}).
           \end{align*}
           This leads to that for all $\theta>0$,
           \begin{align}\label{moment}
           \left(1- \frac{\lambda}{N}(1-\e^{-\theta})\right)^N \leq \Phi_N(\theta) \leq
            \left(1- \frac{\lambda}{N+1}(1-\e^{-\theta})\right)^N.
           \end{align}
           Letting $N\to\infty$, we obtain that for all $\theta>0$, 
           \begin{align*}
           \lim_{N\to\infty}\Phi_N(\theta) = \e^{-\lambda(1-\e^{-\theta})}.
           \end{align*}
           The above display also holds when $\theta \leq 0$ by switching the first and third terms in \eqref{moment}. 
           This verifies \eqref{Poisson:U}.
           
           {\bf Step 2}. We proceed to compare the characteristic functions of $\sum_{k=1}^Ng_N(\mathcal{A}_1(x_{k,N}))$ and
           $\sum_{k=1}^NU_{k, N}$. First, we have for all $\theta\in\R$,
           \begin{align}\label{cov4}
          &\left| \E\left[\e^{{\rm i}\theta \sum_{k=1}^Ng_N(\mathcal{A}_1(x_{k,N}))}\right] 
          - \E\left[\e^{{\rm i}\theta \sum_{k=1}^NU_{k,N}}\right]\right| \nonumber\\
          &\qquad\qquad\qquad= \left| \E\left[\e^{{\rm i} \sum_{k=1}^N\theta g_N(\mathcal{A}_1(x_{k,N}))}\right] 
          -\prod_{k=1}^N \E\left[\e^{{\rm i} \theta g_N(\mathcal{A}_1(x_{k,N}))}\right]\right|.
           \end{align}
           Since the Airy$_1$ process is associated and $g_N$ is a nondecreasing function, 
           the random vector $(g_N(\mathcal{A}_1(x_{1,N})), \ldots, g_N(\mathcal{A}_1(x_{N,N}))$ is associated. Clearly, each
           component of this random vector has finite variance. Hence, by \eqref{Newman3}, we have for all $\theta\in\R$,
           \begin{align}\label{cov3}
           &\left| \E\left[\e^{{\rm i} \sum_{k=1}^N\theta g_N(\mathcal{A}_1(x_{k,N}))}\right] 
          -\prod_{k=1}^N \E\left[\e^{{\rm i} \theta g_N(\mathcal{A}_1(x_{k,N}))}\right]\right| \nonumber\\
          &\qquad\qquad\qquad \leq \frac12 \sum_{\substack{1\leq j, k\leq N\\ j\neq k}} \theta^2 
          \Cov(g_N(\mathcal{A}_1(x_{j,N}))\,, g_N(\mathcal{A}_1(x_{k,N})))\nonumber\\
          &\qquad \qquad\qquad\leq \frac12  [\lip(g_N)]^2 \theta^2\sum_{\substack{1\leq j, k\leq N\\ j\neq k}} 
          \Cov(\mathcal{A}_1(x_{j,N})\,, \mathcal{A}_1(x_{k,N})),
           \end{align}
           where the second inequality is due to \eqref{Newman2}.
           According to \cite[Theorem 1.1]{BBF23}, there exists a positive constant $c'$ such that for all $x, y\in \R$ 
           with $|x-y|>1$,
           \begin{align}\label{airycov}
           \Cov(\mathcal{A}_1(x)\,, \mathcal{A}_1(y)) \leq \e^{c'|x-y|^2}\e^{-\frac43|x-y|^3}.
           \end{align}
          Thus,  a combination of \eqref{cov4}, \eqref{cov3} and \eqref{airycov} yields that for all $\theta\in \R$,
         \begin{align}\label{cha1}
         &\left| \E\left[\e^{{\rm i}\theta \sum_{k=1}^Ng_N(\mathcal{A}_1(x_{k,N}))}\right] 
          - \E\left[\e^{{\rm i}\theta \sum_{k=1}^NU_{k,N}}\right]\right|\nonumber \\
                     &\qquad\qquad\qquad\leq  \frac{\theta^2}{2}[\lip(g_N)]^2 
           \sum_{\substack{1\leq j, k\leq N\\ j\neq k}} \e^{c'|x_{j, N}-x_{k,N}|^2}\, \e^{-\frac43|x_{j, N}-x_{k,N}|^3}
         \end{align}
        for all sufficiently large $N$.

                  Under the assumption \eqref{distance}, there exists $\alpha_1>0$ such that for all sufficiently large $N$
         \begin{align}\label{min}
         \min_{1\leq j\neq k\leq N}|x_{j, N}-x_{k,N}| \geq (3\log N)^{1/3}(1+\alpha_1).
         \end{align}
         Choose and fix $\epsilon \in (0, \frac23)$ such that 
         $
         4(1+\alpha_1)^3 > 2+\frac{2(2+6\epsilon)}{2+3\epsilon}.
         $               
         Furthermore, we can choose a sufficiently small number $\alpha_2\in (0, 1)$ such that
         \begin{align}\label{epalpha}
         3(\frac43-\alpha_2)(1+\alpha_1)^3 > 2+\frac{2(2+6\epsilon)}{2+3\epsilon}.
         \end{align}

           We first derive from \eqref{cha1} that for all sufficiently large $N$,
                   \begin{align*}
         \left| \E\left[\e^{{\rm i}\theta \sum_{k=1}^Ng_N(\mathcal{A}_1(x_{k,N}))}\right] 
          - \E\left[\e^{{\rm i}\theta \sum_{k=1}^NU_{k,N}}\right]\right|
          &\leq  \frac{\theta^2}{2}[\lip(g_N)]^2 
           \sum_{\substack{1\leq j, k\leq N\\ j\neq k}} \e^{-(\frac43-\alpha_2)|x_{j, N}-x_{k,N}|^3}\\
           &\leq  \frac{\theta^2}{2}[\lip(g_N)]^2 N^2
            \e^{-(\frac43-\alpha_2)\min_{1\le j\neq k\leq N}|x_{j, N}-x_{k,N}|^3}.
                    \end{align*}
         Then, by \eqref{lip} and Lemma \ref{lem:diff}, for all sufficiently large $N$,
                  \begin{align*}
         &\left| \E\left[\e^{{\rm i}\theta \sum_{k=1}^Ng_N(\mathcal{A}_1(x_{k,N}))}\right] 
          - \E\left[\e^{{\rm i}\theta \sum_{k=1}^NU_{k,N}}\right]\right|\nonumber \\
           &\qquad\qquad\qquad\leq  C^2\theta^2N^{2}(N+1)^{\frac{2(2+6\epsilon)}{2+3\epsilon}}
             \e^{-(\frac43-\alpha_2)\min_{1\leq j\neq k\leq N}|x_{j, N}-x_{k,N}|^3}\nonumber\\
             &\qquad\qquad\qquad \leq    C^2\theta^2N^{2}(N+1)^{\frac{2(2+6\epsilon)}{2+3\epsilon}}N^{-3(\frac43-\alpha_2)(1+\alpha_1)^3},
             \qquad \text{by \eqref{min}}
                      \end{align*}
          which implies by \eqref{epalpha} that for all $\theta\in\R$,
             \begin{align}\label{cha:diff2}
         \lim_{N\to\infty}\left| \E\left[\e^{{\rm i}\theta \sum_{k=1}^Ng_N(\mathcal{A}_1(x_{k,N}))}\right] 
          - \E\left[\e^{{\rm i}\theta \sum_{k=1}^NU_{k,N}}\right]\right|=0.
          \end{align}
          Therefore, we complete the proof of \eqref{Poisson:g} by \eqref{cha:diff2} and \eqref{Poisson:U}.        
\end{proof}

\begin{proposition}\label{prop:Poisson2}
           Fix $\lambda>0$.
          Let $g_{N-1}$ be defined as in \eqref{g_N} with $N$ replaced by $N-1$ for $N\geq 2$. 
          Let $\{x_{k, N\}_{N\geq1, 1\leq k\leq N}}\subset \R$ satisfy the 
          condition in \eqref{distance}. Then, as $N\to\infty$,
          \begin{align}\label{Poisson:g_N-1}
          \sum_{k=1}^Ng_{N-1}(\mathcal{A}_1(x_{k,N})) \xrightarrow{\rm d} {\rm Poisson}(\lambda).
          \end{align} 
\end{proposition}
\begin{proof}
         The proof follows along the same lines as that of Proposition \ref{prop:Poisson}. One just needs to replace $g_N$ by 
         $g_{N-1}$.
\end{proof}

We are now ready to prove Theorem \ref{th:Poisson}.

\begin{proof}[{{Proof of Theorem \ref{th:Poisson}}}]
           This is an immediate consequence of Propositions \ref{prop:Poisson}, \ref{prop:Poisson2} and \eqref{lowerupper}. 
\end{proof}

\section{Asymptotic behavior of the maximum} \label{sec:max}
In this section, we prove Theorems  \ref{maximum} and \ref{th:A2}. We start with the following estimate on the upper tail probability of the maximum of the Airy$_1$ process.

\begin{proposition}\label{prop:tail}
          There exists a positive constant $C$ such that for all $M\geq2$
          \begin{align}\label{maxtail1}
          \mathrm{P}\left\{\max_{0\leq x\leq 1}\mathcal{A}_1(x)>M\right\} \leq C\, M^{7/4}\e^{(2-\sqrt{2})M+\sqrt{2M}} \e^{- \frac{4}{3}\sqrt{2}M^{3/2} },
          \end{align}
          and it holds that
          \begin{align}\label{maxtail2}
          \lim_{M\to\infty} \frac{\log \mathrm{P}\left\{\max_{0\leq x\leq 1}\mathcal{A}_1(x)>M\right\}}{M^{3/2}}=-\frac43\sqrt{2}.
          \end{align}
\end{proposition}
\begin{proof}
          Because $\mathrm{P}\{\max_{0\leq x\leq 1}\mathcal{A}_1(x)>M\} \geq \mathrm{P}\{\mathcal{A}_1(0)>M\}$, the equality \eqref{maxtail2} follows immediately from the estimate \eqref{maxtail1} and the asymptotic behavior of the GOE Tracy-Widom distribution 
          (see \cite[Theorem 1]{DuV13}). In order to prove \eqref{maxtail1}, we recall from \cite[Theorem 4]{QuR13} that
          \begin{align}\label{det}
          \mathrm{P}\left\{\max_{0\leq x\leq 1}\mathcal{A}_1(x)>M\right\}&=1- \mathrm{P}\left\{\max_{0\leq x\leq 1}\mathcal{A}_1(x)\leq M\right\}
          =1- \mathrm{P}\left\{\mathcal{A}_1(x)\leq M, \text{for all $x\in [0, 1]$}\right\}\nonumber\\
          &=1- \det\left( I - B_0+ \Lambda_{[0, 1]}^g\e^{-\Delta}B_0 \right)_{L^2(\R)},
          \end{align}
          where in the Fredholm determinant, the kernel $\Lambda_{[0, 1]}^g$ has the formula given by (1.15) of \cite{QuR13} with $g\equiv M$ on $[0, 1]$ and the kernel $B_0$ is 
          given in terms of the Airy function (see the equality below (1.5) of \cite{QuR13}).

          We will adopt the method in the proof of \cite[Proposition 2.3(a)]{QuR13} to estimate the above tail probability. 
          We introduce the operator $U$ (depending on $M$) as
          \begin{align*}
          Uf(x)= \e^{-\sqrt{2M}x}f(x). 
          \end{align*}
          Let $\varphi(z)= \sqrt{1+z^2}$ and write
          \begin{align*}
          V(x,z)&= \left(\e^{\Delta} - \Lambda_{[0,1]}^g\right)(x, z)\e^{-\sqrt{2M}x}\varphi(z)\e^{-\sqrt{2M}z},\\
          W(z, y)&= \left(\e^{-\Delta}B_0\right)(z,y)\varphi(z)^{-1}\e^{\sqrt{2M}z}\e^{\sqrt{2M}y}.
          \end{align*}
          Then 
          \begin{align*}
          U\left(B_0- \Lambda_{[0, 1]}^g\e^{-\Delta}B_0\right)U^{-1}= VW. 
          \end{align*}
          We will estimate the Hilbert-Schmidt norms $\|V\|_2$ and $\|W\|_2$ and the use the inequality $\|VW\|_1\leq \|V\|_2\|W\|_2$ to give an upper bound on the trace norm 
          $\|VW\|_1$.
          
          We first estimate $\|W\|_2$. Denote by ${\rm Ai}$ the Airy function and $\|\varphi^{-1}\|_2$ the $L^2(\R)$-norm of $\varphi^{-1}$. Using (1.9) of \cite{QuR13}, we write
          \begin{align*}
          \|W\|_2^2&= \int_{\R^2}\mathrm{d} x \mathrm{d} y  \frac{\e^{-\frac43+2(\sqrt{2M}-1)(x+y)}}{\varphi(x)^2}{\rm Ai}(x+y+1)^2
          =\|\varphi^{-1}\|_2^2 \int_{-\infty}^{\infty} \e^{-\frac43+ 2(\sqrt{2M}-1)y}{\rm Ai}(y+1)^2\mathrm{d} y\\
          &=\|\varphi^{-1}\|_2^2\, \e^{-\frac43-2(\sqrt{2M}-1)}\int_{-\infty}^{\infty} \e^{2(\sqrt{2M}-1)y}{\rm Ai}(y)^2\mathrm{d} y.
          \end{align*}
          By the estimate (2.19) of \cite{QuR13} on the Airy function, we have for all $M\geq2$
          \begin{align}\label{eq:W}
          \|W\|_2^2&\leq C\|\varphi^{-1}\|_2^2\, \e^{-\frac43-2(\sqrt{2M}-1)}\left(\int_{-\infty}^0 \e^{2(\sqrt{2M}-1)y} \mathrm{d} y + \int_0^\infty \e^{2(\sqrt{2M}-1)y-\frac43y^{3/2}}\mathrm{d} y\right)\nonumber\\
          &\leq c_1' (\sqrt{2M}-1)^2\e^{-2(\sqrt{2M}-1)}\e^{\frac23(\sqrt{2M}-1)^3},
          \end{align}
          where the second inequality holds by Lemma \ref{lem:SD}.

          We proceed to estimate $\|V\|_2$. We use the formula (1.15) of \cite{QuR13} to write
          \begin{align*}
          V(x,y)= \frac{\varphi(y)}{\sqrt{4\pi}} \e^{-\frac{(x-y)^2}{4} -\sqrt{2M}(x+y)} \mathrm{P}_{\hat{b}(0)=x, \hat{b}(1)=y} \left(\hat{b}(t)\geq M \,\,\text{for some $t\in [0, 1]$}\right),
          \end{align*}
          where $\hat{b}$ denotes a Brownian bridge from $x$ at time $0$ to $y$ at time $1$ with diffusion constant $2$. This probability equals $\e^{-(x-M)(y-M)}$ 
          if $x\leq M, y\leq M$ and $1$ otherwise (see line one on page 621 of \cite{QuR13} or page 69 in \cite{BoS02}). Hence,
          \begin{align*}
          \|V\|_2^2&= \frac{1}{4\pi} \int_{(-\infty, M]^2}\mathrm{d} x \mathrm{d} y\, (1+y^2)\left[\e^{-\frac{(x-y)^2}{4} -\sqrt{2M}(x+y)-(x-M)(y-M) }\right]^2\\
          &\quad +  \frac{1}{4\pi} \int_{\R^2\setminus(-\infty, M]^2}\mathrm{d} x \mathrm{d} y\, (1+y^2)\left[\e^{-\frac{(x-y)^2}{4} -\sqrt{2M}(x+y) }\right]^2\\
          &:=I_1+I_2+I_3+I_4,
          \end{align*}
          where
          \begin{align*}
          I_1&= \frac{1}{4\pi} \int_{(-\infty, M]^2}\mathrm{d} x \mathrm{d} y\, \e^{-\frac{(x-y)^2}{2} -2\sqrt{2M}(x+y)-2(x-M)(y-M) },\\
           I_2&=\frac{1}{4\pi} \int_{(-\infty, M]^2}\mathrm{d} x \mathrm{d} y\, y^2\e^{-\frac{(x-y)^2}{2} -2\sqrt{2M}(x+y)-2(x-M)(y-M) },\\
           I_3&=\frac{1}{4\pi} \int_{\R^2\setminus(-\infty, M]^2}\mathrm{d} x \mathrm{d} y\, \e^{-\frac{(x-y)^2}{2} -2\sqrt{2M}(x+y) },\\
           I_4&=\frac{1}{4\pi} \int_{\R^2\setminus(-\infty, M]^2}\mathrm{d} x \mathrm{d} y\, y^2\e^{-\frac{(x-y)^2}{2} -2\sqrt{2M}(x+y) }.
          \end{align*}
          Using change of variable [$x\to x+M, y\to y+M$],
          \begin{align}\label{eq:I1}
          I_1&= \frac{\e^{-4\sqrt{2}M^{3/2}}}{4\pi} \int_{(-\infty, 0]^2}\mathrm{d} x \mathrm{d} y\, \e^{-\frac{(x+y)^2}{2} -2\sqrt{2M}(x+y)}
          \leq c_2' \e^{8M}\e^{-4\sqrt{2}M^{3/2}}
          \end{align}
          for all $M\geq2$, where the inequality holds by Lemma \ref{lem:SD2}.
          Similarly,  using the inequality $(a+b)^2 \leq 2a^2+2b^2$,
          \begin{align}\label{eq:I2}
          I_2&=\frac{\e^{-4\sqrt{2}M^{3/2}}}{4\pi} \int_{(-\infty, 0]^2}\mathrm{d} x \mathrm{d} y\, (y+M)^2\e^{-\frac{(x+y)^2}{2} -2\sqrt{2M}(x+y)} \nonumber\\
          &\leq 2M^2I_1 + \frac{\e^{-4\sqrt{2}M^{3/2}}}{2\pi} \int_{(-\infty, 0]^2}\mathrm{d} x \mathrm{d} y\, y^2\e^{-\frac{(x+y)^2}{2} -2\sqrt{2M}(x+y)} 
          \leq c_3' M^{2}\e^{8M}\e^{-4\sqrt{2}M^{3/2}}
          \end{align}
          for all $M\geq2$, where the second inequality holds by \ref{lem:SD3}
          
          We move on to estimate $I_3$ and $I_4$.  We use change of variable  [$x\to x+M, y\to y+M$] to see that
          \begin{align}\label{eq:I3}
          I_3&=\frac{\e^{-4\sqrt{2}M^{3/2}}}{4\pi} \int_{\R^2\setminus(-\infty, 0]^2}\mathrm{d} x \mathrm{d} y\, \e^{-\frac{(x-y)^2}{2} -2\sqrt{2M}(x+y) }
           \leq c_4' \sqrt{M}\e^{4M}\e^{-4\sqrt{2}M^{3/2}}
          \end{align}
          for all $M\geq2$, where the inequality holds by Lemma \ref{lem:SD4}.  It remains to estimate $I_4$. By change of variable  [$x\to x+M, y\to y+M$], 
          \begin{align}\label{eq:I4}
          I_4&=\frac{\e^{-4\sqrt{2}M^{3/2}}}{4\pi} \int_{\R^2\setminus(-\infty, 0]^2}\mathrm{d} x \mathrm{d} y\, (y+M)^2\e^{-\frac{(x-y)^2}{2} -2\sqrt{2M}(x+y) }   \nonumber\\
          &\leq 2M^2I_3 + \frac{\e^{-4\sqrt{2}M^{3/2}}}{2\pi} \int_{\R^2\setminus(-\infty, 0]^2}\mathrm{d} x \mathrm{d} y\, y^2\e^{-\frac{(x-y)^2}{2} -2\sqrt{2M}(x+y) }    
          \leq c_5' M^{5/2}\e^{4M} \e^{-4\sqrt{2}M^{3/2}}
          \end{align}
                    for all $M\geq2$, 
          where the second inequality holds by \eqref{eq:I3} and Lemma \ref{lem:SD5}.

          Now we combine the estimates \eqref{eq:W}, \eqref{eq:I1}, \eqref{eq:I2},\eqref{eq:I3} and \eqref{eq:I4} to see that there exists a constant $C'>0$ such that for all $M\geq2$,
          \begin{align*}
          \|V\|_2^2\|W\|_2^2 &\leq C'\, M^{7/2}\e^{(8-2\sqrt{2})M} \e^{\frac23(\sqrt{2M}-1)^3 - 4\sqrt{2}M^{3/2}}
          =C'\, \e^{-2/3}M^{7/2}\e^{(4-2\sqrt{2})M+2\sqrt{2M}} \e^{- \frac{8}{3}\sqrt{2}M^{3/2} }, 
          \end{align*}
          which implies that for all $M\geq2$,
          \begin{align}\label{VW}
          \|VW\|_1 \leq \|V\|_2\|W\|_2 \leq \sqrt{C'} \e^{-1/3}M^{7/4}\e^{(2-\sqrt{2})M+\sqrt{2M}} \e^{- \frac{4}{3}\sqrt{2}M^{3/2} }.
          \end{align}
          Thus, $U(B_0- \Lambda_{[0, 1]}^g\e^{-\Delta}B_0)U^{-1}$ is a trace class operator on $L^2(\R)$ and we deduce from \eqref{det} that
          \begin{align}\label{det2}
          \mathrm{P}\left\{\max_{0\leq x\leq 1}\mathcal{A}_1(x)>M\right\}
             &=1- \det\left( I - U(B_0- \Lambda_{[0, 1]}^g\e^{-\Delta}B_0)U^{-1} \right)_{L^2(\R)}\nonumber\\
             &=1- \det\left( I - VW \right)_{L^2(\R)}
             \leq \|VW\|_1\e^{\|VW\|_1 +1},
          \end{align}
          where the inequality follows from (3.1) of \cite{QuR13}.  Therefore, we conclude from \eqref{VW} and \eqref{det2} that there exists $C>0$ such that for all $M\geq2$,
          \begin{align*}
          \mathrm{P}\left\{\max_{0\leq x\leq 1}\mathcal{A}_1(x)>M\right\} \leq C\, M^{7/4}\e^{(2-\sqrt{2})M+\sqrt{2M}} \e^{- \frac{4}{3}\sqrt{2}M^{3/2} }.
          \end{align*}
          
          The proof is complete.          
\end{proof}

\begin{remark}
          We remark that another asymptotics on the probability that the Airy$_1$ process stays below a given
threshold $c$ for a time span of length $L$ is investigated in  \cite{FeL24}.
\end{remark}

Before we prove Theorem \ref{maximum}, we present a probability inequality for the Airy$_1$ process: there exists a constant $K>0$ such that for all $x, y\in\R$,
          \begin{align}\label{eq:prob}
          &\sup_{s, t\in \R}\left(\mathrm{P}\left\{\mathcal{A}_1(x) \leq s, \mathcal{A}_1(y)\leq t\right\} - 
          \mathrm{P}\left\{\mathcal{A}_1(x) \leq s\right\}\mathrm{P}\left\{\mathcal{A}_1(y)\leq t\right\}  \right)           \leq K\left[\Cov(\mathcal{A}_1(x)\,,
          \mathcal{A}_1(y))\right]^{1/3}.
          \end{align}
The above inequality follows from \cite[(6.2.20)]{PR12} (see also \cite[Theorem 6.2.15]{PR12})       because the Airy$_1$ process is associated and its one-point distribution, i.e., the GOE Tracy-Widom distribution has bounded and continuous density and finite second moment.

We are now ready to prove Theorem \ref{maximum}.

\begin{proof}[{Proof of Theorem \ref{maximum}}]
          We first show that almost surely
          \begin{align}\label{asymlow}
          \liminf_{N\to\infty} \frac{\max_{0\leq x\leq N}\mathcal{A}_1(x)}{(\log N)^{2/3}} \geq \frac12\left(\frac32\right)^{2/3}.
          \end{align}
          Let $\beta$ be a positive number that is strictly less than $\frac32$. Choose and fix $a, \epsilon \in (0, 1)$
           ($a$ close to $1$ and $\epsilon$ close to $0$) such that
          \begin{align}\label{exponent}
          \beta < \frac{3a}{2(1+\epsilon)}.
          \end{align}
          Define $x_j= jN/\floor*{N^a}$ 
          for $j=1, \ldots, \floor*{N^a}$. Here the notation $\floor*{x}$ denotes the largest integer which is less than or equal to $x$. By the stationarity of the Airy$_1$ process,
          \begin{align*}
          &\mathrm{P}\left\{\max_{0\leq x\leq N}\mathcal{A}_1(x) \leq  \frac12(\beta\log N)^{2/3}\right\}
           \leq \mathrm{P}\left\{\max_{1\leq j\leq \floor*{N^a}}\mathcal{A}_1(x_j) \leq  \frac12(\beta\log N)^{2/3}\right\}\\
          &\qquad\qquad=  \mathrm{P}\left\{\max_{1\leq j\leq \floor*{N^a}}\mathcal{A}_1(x_j) \leq  \frac12(\beta\log N)^{2/3}\right\}
          - \prod_{j=1}^{ \floor*{N^a}} \mathrm{P}\left\{\mathcal{A}_1(x_j) \leq  \frac12(\beta\log N)^{2/3}\right\}\\
          &\qquad\qquad\qquad \qquad\qquad\qquad\qquad\qquad
          + \left(\mathrm{P}\left\{\mathcal{A}_1(0) \leq  \frac12(\beta\log N)^{2/3}\right\} \right)^{ \floor*{N^a}}.
          \end{align*}
          Since the  Airy$_1$ process is associated, by \eqref{FKG2}, we see that
          \begin{align*}
          &\mathrm{P}\left\{\max_{0\leq x\leq N}\mathcal{A}_1(x) \leq  \frac12(\beta\log N)^{2/3}\right\} \\
          &\qquad \leq \sum_{1\leq j<k \leq  \floor*{N^a}}
          \bigg(\mathrm{P}\left\{\mathcal{A}_1(x_j)\leq  \frac12(\beta\log N)^{2/3}, \mathcal{A}_1(x_k)\leq  \frac12(\beta\log N)^{2/3}\right\}\\
          &\qquad\qquad\qquad\qquad\qquad\qquad\qquad
          -\mathrm{P}\left\{\mathcal{A}_1(x_j)\leq  \frac12(\beta\log N)^{2/3}\right\}\mathrm{P}\left\{\mathcal{A}_1(x_k)\leq  \frac12(\beta\log N)^{2/3}\right\}
          \bigg)\\
          &\qquad\qquad+ \left(\mathrm{P}\left\{\mathcal{A}_1(0) \leq  \frac12(\beta\log N)^{2/3}\right\} \right)^{ \floor*{N^a}}.
          \end{align*}
          We apply \eqref{eq:prob} to deduce that
          \begin{align}
          &\mathrm{P}\left\{\max_{0\leq x\leq N}\mathcal{A}_1(x) \leq  \frac12(\beta\log N)^{2/3}\right\} \nonumber\\
          &\qquad \leq K\sum_{1\leq j<k \leq  \floor*{N^a}} \left[\Cov(\mathcal{A}_1(x_j)\,,\mathcal{A}_1(x_k))\right]^{1/3}
          +\left(\mathrm{P}\left\{\mathcal{A}_1(0) \leq  \frac12(\beta\log N)^{2/3}\right\} \right)^{ \floor*{N^a}}. \label{cov:A1}
          \end{align}

          We can assume that $N$ is sufficiently large such that $N/  \floor*{N^a} >\max\{1, 3c'\}$ with $c'$ being the number in 
          \eqref{airycov}, $\floor*{N^a} > N^a/2$ and $(\beta\log N)^{2/3}>C_2$, where $C_2$ is the constant in \eqref{TWtail}. 
                     Hence,  
          by \eqref{airycov}, 
                    \begin{align*}
          &\mathrm{P}\left\{\max_{0\leq x\leq N}\mathcal{A}_1(x) \leq \frac12(\beta\log N)^{2/3}\right\} 
           \leq K\sum_{1\leq j<k \leq  \floor*{N^a}} \e^{\frac{c'}{3}|x_j-x_k|^2}\e^{-\frac49|x_j-x_k|^3}
          +F_1\left((\beta\log N)^{2/3}\right)^{ \floor*{N^a}}\\
          &\qquad\qquad\qquad \leq K\sum_{1\leq j<k \leq  \floor*{N^a}}\e^{-\frac13|x_j-x_k|^3} + \left(1-\e^{-\frac23(1+\epsilon)\beta\log N}\right)^{ \floor*{N^a}}\\
          & \qquad\qquad\qquad \leq KN^{2a}\e^{-\frac13N^{3(1-a)}} + \e^{-N^{-\frac23(1+\epsilon)\beta}\floor*{N^a}} 
          \leq KN^{2a}\e^{-\frac13N^{3(1-a)}} + \e^{-\frac12N^{-\frac23(1+\epsilon)\beta+a}},
          \end{align*}
          where in the second inequality, we use the first inequality in \eqref{TWtail}, and in the third inequality, we use the fact
          $1-x\leq \e^{-x}$ for all $x\geq0$. Since $-\frac23(1+\epsilon)\beta+a >0$ by \eqref{exponent}, we conclude that 
          \begin{align*}
          \sum_{N=1}^{\infty} \mathrm{P}\left\{\max_{0\leq x\leq N}\mathcal{A}_1(x) \leq \frac12(\beta\log N)^{2/3}\right\}<\infty. 
          \end{align*}
          Therefore, Borel-Cantelli's lemma implies that almost surely,
          \begin{align*}
          \liminf_{N\to\infty} \frac{\max_{0\leq x\leq N}\mathcal{A}_1(x)}{(\log N)^{2/3}} \geq \frac12\beta^{2/3}.
          \end{align*}
          The above liminf is taken along integers. By the monotonicity of $\max_{0\leq x\leq N}\mathcal{A}_1(x)$ 
          in $N$, we see that the preceding also 
          holds when the liminf is taken
          along real numbers. Finally, by letting $\beta \uparrow \frac32$,  we obtain \eqref{asymlow}.

          We turn to proving that almost surely
          \begin{align}\label{asymup}
          \limsup_{N\to\infty} \frac{\max_{0\leq x\leq N}\mathcal{A}_1(x)}{(\log N)^{2/3}} \leq \frac12\left(\frac32\right)^{2/3}.
          \end{align}
           Let $\gamma$ be a positive number that is strictly larger than $\frac{3}{4\sqrt{2}}$. 
          Choose and fix $\epsilon \in (0, 1)$ such that 
          \begin{align}\label{eq:exp1}
          \gamma > \frac{1}{\frac{4}{3}\sqrt{2}- \epsilon}.
          \end{align}
          For this fixed $\epsilon\in (0, 1)$, Proposition \ref{prop:tail} ensures that there exists constant $C_1>0$ depending on $\epsilon$ such that 
          \begin{align}\label{tail:ep}
          \mathrm{P}\left\{ \max_{0\leq x\leq 1}\mathcal{A}_1(x) >M\right\} \leq \e^{-(\frac43\sqrt{2}-\epsilon)M^{3/2}}, \quad \text{for all $M\geq C_1$}. 
          \end{align}
          Assume for the moment that $N$ is a sufficiently large integer such that $(\gamma\log N)^{2/3} \geq C_1$ with $C_1$ being the constant in \eqref{tail:ep}.
          By the stationarity of the Airy$_1$ process,
          \begin{align*}
          \mathrm{P}\left\{ \max_{0\leq x\leq N}\mathcal{A}_1(x) > (\gamma\log N)^{2/3} \right\} &= \mathrm{P}\left\{ \cup_{j=1}^N \left\{\max_{j-1\leq x\leq j}\mathcal{A}_1(x) >
           (\gamma\log N)^{3/2}\right\}\right\}\\
           & \leq N \mathrm{P}\left\{ \max_{0\leq x\leq 1}\mathcal{A}_1(x) > (\gamma\log N)^{2/3} \right\} \\
           & \leq N \e^{-(\frac43\sqrt{2}-\epsilon)\gamma \log N} = N^{1- (\frac43\sqrt{2}-\epsilon)\gamma},
          \end{align*}
          where the second inequality follows from \eqref{tail:ep}.  Since $ (\frac43\sqrt{2}-\epsilon)\gamma >1$ by \eqref{eq:exp1}, letting $N=2^n$, we obtain that 
          \begin{align*}
          \sum_{n=1}^\infty\mathrm{P}\left\{ \max_{0\leq x\leq 2^n}\mathcal{A}_1(x) > (\gamma\log 2^n)^{2/3} \right\}<\infty,
          \end{align*}
          which implies that almost surely
          \begin{align*}
          \limsup_{n\to\infty} \frac{\max_{0\leq x\leq 2^n}\mathcal{A}_1(x)}{(\log 2^n)^{2/3}} \leq \gamma^{2/3}
          \end{align*}
          by Borel-Cantelli's lemma. 
          Since the quantity $\max_{0\leq x\leq N}\mathcal{A}_1(x)$ is monotone in $N$, we obtain that almost surely
           \begin{align*}
          \limsup_{N\to\infty} \frac{\max_{0\leq x\leq N}\mathcal{A}_1(x)}{(\log N)^{2/3}} \leq \gamma^{2/3}. 
          \end{align*}
          Letting $\gamma \downarrow \frac{3}{4\sqrt{2}}$, we prove \eqref{asymup}.

          Theorem \ref{maximum} follows from \eqref{asymlow} and \eqref{asymup}.
\end{proof}

We proceed to prove Theorem \ref{th:A2}.  Recall from  \cite[Theorem 1]{DuV13} that  for $\epsilon\in (0, 1)$, there exists a positive constant $\tilde{C}_2$ depending on $\epsilon$ such that
          \begin{align}\label{tail:A2}
           \e^{-(\frac43+\epsilon)s^{3/2}} \leq 
            \mathrm{P}\{\mathcal{A}_2(0) \geq s\} \leq \e^{-(\frac43-\epsilon)s^{3/2}}, \quad \text{for all $s\geq \tilde{C}_2$}.
          \end{align}
Analogous to \eqref{eq:prob}, using the association property of the Airy$_2$ process and the fact that  the GUE Tracy-Widom distribution has bounded continuous 
          probability density function and finite second moment, we see from \cite[(6.2.20)]{PR12} (see also \cite[Theorem 6.2.15]{PR12}) that there exists a constant $K>0$ such that for all $x, y\in\R$,
          \begin{align}\label{eq:prob2}
          &\sup_{s, t\in \R}\left(\mathrm{P}\left\{\mathcal{A}_2(x) \leq s, \mathcal{A}_2(y)\leq t\right\} - 
          \mathrm{P}\left\{\mathcal{A}_2(x) \leq s\right\}\mathrm{P}\left\{\mathcal{A}_2(y)\leq t\right\}  \right)           \leq K\left[\Cov(\mathcal{A}_2(x)\,,
          \mathcal{A}_2(y))\right]^{1/3}.
          \end{align}

Now we start the proof of Theorem \ref{th:A2}.

\begin{proof}[{Proof of Theorem \ref{th:A2}}]
          We first prove the lower bound in  \eqref{eq:A2}. The strategy is similar to that of the proof of \eqref{eq:max}, using the 
          association property of the Airy$_2$ process, the asymptotic behavior of the covariance of the Airy$_2$ process 
          and the tail distribution of the GUE Tracy-Widom distribution. 
          
          Let $\beta$ be a positive number that is strictly less than $\frac14$. Choose and fix $a\in (0, \frac13)$,
           $\epsilon \in (0, 1)$ such that
          \begin{align}\label{exponent2}
          \beta < \frac{a}{\frac43+\epsilon}.
          \end{align}
         Define $x_j= jN/\floor*{N^a}$ 
          for $j=1, \ldots, \floor*{N^a}$. Then by the stationarity of the Airy$_2$ process,
          \begin{align*}
          &\mathrm{P}\left\{\max_{0\leq x\leq N}\mathcal{A}_2(x) \leq  (\beta\log N)^{2/3}\right\}
           \leq \mathrm{P}\left\{\max_{1\leq j\leq \floor*{N^a}}\mathcal{A}_2(x_j) \leq  (\beta\log N)^{2/3}\right\}\\
          &\qquad\qquad=  \mathrm{P}\left\{\max_{1\leq j\leq \floor*{N^a}}\mathcal{A}_2(x_j) \leq  (\beta\log N)^{2/3}\right\}
          - \prod_{j=1}^{ \floor*{N^a}} \mathrm{P}\left\{\mathcal{A}_2(x_j) \leq  (\beta\log N)^{2/3}\right\}\\
          &\qquad\qquad\qquad \qquad\qquad\qquad\qquad\qquad
          + \left(\mathrm{P}\left\{\mathcal{A}_2(0) \leq  (\beta\log N)^{2/3}\right\} \right)^{ \floor*{N^a}}.
          \end{align*}
          We can assume that $N$ is sufficiently large such that $\floor*{N^a} > N^a/2$ and $(\beta\log N)^{2/3}>C'_2$, 
          where $C'_2$ is the constant in \eqref{tail:A2}. Since the  Airy$_2$ process is associated (see Proposition \ref{A2:association}), by \eqref{FKG2}, we see that
          \begin{align}
          &\mathrm{P}\left\{\max_{0\leq x\leq N}\mathcal{A}_2(x) \leq  (\beta\log N)^{2/3}\right\}\nonumber \\
          &\qquad \leq \sum_{1\leq j<k \leq  \floor*{N^a}}
          \bigg(\mathrm{P}\left\{\mathcal{A}_2(x_j)\leq  (\beta\log N)^{2/3}, \mathcal{A}_2(x_k)\leq  (\beta\log N)^{2/3}\right\} \nonumber\\
          &\qquad\qquad\qquad\qquad\qquad\qquad\qquad
          -\mathrm{P}\left\{\mathcal{A}_2(x_j)\leq  (\beta\log N)^{2/3}\right\}\mathrm{P}\left\{\mathcal{A}_2(x_k)\leq  (\beta\log N)^{2/3}\right\}
          \bigg) \nonumber\\
          &\qquad\qquad+ \left(1-\mathrm{P}\left\{\mathcal{A}_2(0) >  (\beta\log N)^{2/3}\right\} \right)^{ \floor*{N^a}} \nonumber\\
          &\qquad \leq K  \sum_{1\leq j<k \leq  \floor*{N^a}}\left[\Cov(\mathcal{A}_2(x_j)\,, \mathcal{A}_2(x_k))\right]^{1/3}+
          \left(1-\mathrm{P}\left\{\mathcal{A}_2(0) >  (\beta\log N)^{2/3}\right\} \right)^{ \floor*{N^a}}, \label{cov:A2}
          \end{align}
          where the second inequality holds by \eqref{eq:prob2}. Moreover, since the decay rate of the covariance 
          $\Cov(\mathcal{A}_2(x),$ $\mathcal{A}_2(0))$ is $x^{-2}$ as $x\to\infty$ (see \cite{Wid04}), we deduce that
          \begin{align*}
          \mathrm{P}\left\{\max_{0\leq x\leq N}\mathcal{A}_2(x) \leq  (\beta\log N)^{2/3}\right\}& \leq K'   \sum_{1\leq j<k \leq  \floor*{N^a}} \frac{1}{|x_k-x_j|^{2/3}} + 
          \left(1-\e^{-(\frac43+\epsilon)\beta\log N}\right)^{  \floor*{N^a}}\\
          &\leq K'N^{\frac23(a-1)}\sum_{1\leq j<k \leq  \floor*{N^a}} \frac{1}{|k-j|^{2/3}} + 
          \e^{-  \floor*{N^a} N^{-(\frac43+\epsilon)\beta}},
          \end{align*}
          where we have used \eqref{tail:A2} in the first inequality and $1-x\leq \e^{-x}$ for all $x\geq0$ in the second inequality.
          Because for integer $m\geq2$,
          \begin{align*}
          \sum_{k=2}^m\sum_{j=1}^{k-1}\frac{1}{(k-j)^{2/3}}&=           \sum_{k=2}^m\sum_{j=1}^{k-1}\frac{1}{j^{2/3}} \leq 
          \sum_{k=2}^{m}\int_0^{k-1}\frac{\mathrm{d} x}{x^{2/3}}=3\sum_{k=2}^{m}(k-1)^{1/3} 
          \leq \frac94
          \int_0^m x^{1/3}\mathrm{d} x = \frac94 m^{4/3},
          \end{align*}
          we conclude that
                    \begin{align*}
          \mathrm{P}\left\{\max_{0\leq x\leq N}\mathcal{A}_2(x) \leq  (\beta\log N)^{2/3}\right\}& \leq \frac{9K'}{4}N^{\frac23(a-1)+\frac43a} + \e^{-\frac12 N^{a-(\frac43+\epsilon)\beta}}.   
           \end{align*}
           Letting $N=2^{n}$, we obtain that
           \begin{align*}
          \mathrm{P}\left\{\max_{0\leq x\leq 2^n}\mathcal{A}_2(x) \leq  (\beta\log 2^n)^{2/3}\right\}& \leq \frac{9K'}{4}2^{-\frac23(1-3a)n} + \e^{-\frac12 2^{(a-(\frac43+\epsilon)\beta)n}}.   
         \end{align*}
         Since $a<\frac13$ and $a>(\frac43+\epsilon)\beta$ (see \eqref{exponent2}), we see that
         \begin{align*}
         \sum_{n=1}^{\infty}\mathrm{P}\left\{\max_{0\leq x\leq 2^n}\mathcal{A}_2(x) \leq  (\beta\log 2^n)^{2/3}\right\}<\infty. 
         \end{align*}
         Hence, by Borel-Cantelli's lemma, we have almost surely, 
         \begin{align*}
         \liminf_{n\to\infty}\frac{\max_{0\leq x\leq 2^n}\mathcal{A}_2(x)}{(\log 2^n)^{2/3}} \geq \beta^{2/3}.
         \end{align*} 
         A monotonicity argument yields that almost surely,
         \begin{align*}
         \liminf_{N\to\infty}\frac{\max_{0\leq x\leq N}\mathcal{A}_2(x)}{(\log N)^{2/3}} \geq \beta^{2/3}.
         \end{align*} 
         Letting $\beta\uparrow \frac14$, we obtain the lower bound in \eqref{eq:A2}.

         We proceed to prove the upper bound in \eqref{eq:A2}. By \cite[Corollary 1.3]{CHH23} and stationarity of the Airy$_2$
         process, there exists a positive constant $s_0$ such that
         \begin{align}\label{uniformA2tail}
         \sup_{x\in\R} \mathrm{P}\left(\sup_{0\leq y\leq 2}|\mathcal{A}_2(x+y)-\mathcal{A}_2(x)| \geq s\right) \leq \e^{-\frac{s^2}{16}}, 
         \quad \text{for all $s>s_0$}.
         \end{align}
         Let $\gamma$ be a positive number that is strictly larger than $\frac34$.  Choose and fix $\delta\in (0, 1)$ and $\epsilon\in 
         (0, 1)$ such that
         \begin{align}\label{exponent3}
         (\frac{4}{3}-\epsilon)\delta^{\frac32}\gamma >1. 
         \end{align}
         Assume $N$ is sufficiently large so that 
         \begin{align*}
         \delta(\gamma\log N)^{2/3}> \tilde{C}_2 \quad \text{and}\quad  (1-\delta)(\gamma\log N)^{2/3} >s_0,
         \end{align*}
         where the constants $\tilde{C}_2$ and $s_0$ are given in \eqref{tail:A2} and \eqref{uniformA2tail} respectively. By the 
         triangle inequality and stationarity of the Airy$_2$ process, 
         \begin{align*}
         &\mathrm{P}\left(\max_{0\leq x\leq 2N}\mathcal{A}_2(x) \geq (\gamma\log (2N))^{2/3}\right) \leq 
         \mathrm{P}\left(\cup_{j=0}^{N-1}\left\{\mathcal{A}_2(2j) \geq \delta(\gamma\log (2N))^{2/3}\right\}\right)\\
         &\qquad\qquad\qquad\qquad\qquad
          +\mathrm{P}\left(\cup_{j=1}^N\sup_{y\in [2(j-1), 2j]}|\mathcal{A}_2(y)- \mathcal{A}_2(2(j-1))| \geq (1-\delta)(\gamma\log (2N))^{2/3}\right)\\
          &\qquad\quad \leq N\mathrm{P}\left(\mathcal{A}_2(0) \geq \delta(\gamma\log (2N))^{2/3}\right) + 
          N\mathrm{P}\left(\sup_{y\in [0,\, 2]}|\mathcal{A}_2(y)- \mathcal{A}_2(0)| \geq (1-\delta)(\gamma\log (2N))^{2/3}\right)\\
          &\qquad\quad \leq N\e^{-(\frac43-\epsilon)\delta^{\frac32}\gamma\log (2N)}+ N\e^{-\frac{(1-\delta)^2}{16}(\gamma\log (2N))^{4/3}} = c_0 N^{-(\frac43-\epsilon)\delta^{\frac32}\gamma+1}+ N\e^{-\frac{(1-\delta)^2}{16}(\gamma\log (2N))^{4/3}},
         \end{align*}
         where the last inequality follows from \eqref{tail:A2} and \eqref{uniformA2tail}. 
         Letting $N=2^n$ and because of \eqref{exponent3}, we obtain that
         \begin{align*}
         \sum_{n=1}^{\infty}\mathrm{P}\left(\max_{0\leq x\leq 2^{n+1}}\mathcal{A}_2(x) \geq (\gamma\log 2^{n+1})^{2/3}\right)<\infty,
         \end{align*}
         which implies that almost surely,
         \begin{align*}
         \limsup_{n\to\infty} \frac{\max_{0\leq x\leq 2^{n+1}}\mathcal{A}_2(x)}{\log 2^{n+1}} \leq \gamma^{2/3},
         \end{align*}
         thanks to Borel-Cantelli's lemma. Appealing to the monotonicity argument again, we conclude that almost surely, 
          \begin{align*}
         \limsup_{N\to\infty} \frac{\max_{0\leq x\leq N}\mathcal{A}_2(x)}{\log N} \leq \gamma^{2/3}.
         \end{align*}
         Letting $\gamma\downarrow\frac34$, we obtain the upper bound in \eqref{eq:A2}. 
         
         The proof is complete.
\end{proof}

\begin{remark}
         The approach for the lower bounds in Theorems \ref{maximum} and \ref{th:A2} essentially depends on the decay rate of the covariance of the Airy processes; see 
         \eqref{cov:A1} and \eqref{cov:A2}. The exponential decay rate of the Airy$_1$ process established in \cite{BBF23} leads to the desired lower bound in Theorem  \ref{maximum}, while the polynomial decay rate of the Airy$_2$ process established in \cite{Wid04} results in a non-optimal lower bound in Theorem \ref{th:A2}.  The precise limit of  the growth of maximum of the Airy$_2$ processes is proved in a recent work \cite[Theorem 1.1(i)]{BB24} using last passage percolation after this
paper appeared as preprint, and a different proof of Theorem \ref{maximum} is given in \cite[Theorem 1.2(i)]{BB24}.  See also \cite[Theorem 1.1(ii) and Theorem 1.2(ii)]{BB24} for the asymptotics of minimum of Airy processes. 
\end{remark}

\begin{appendix}
\section{Appendix}
We include in this section a few technical lemmas which are used in the proof of Proposition \ref{prop:tail}.

\begin{lemma}\label{lem:SD}
          There exists a constant $c_1>0$ such that for all $\theta\geq 1$
          \begin{align}\label{SD1}
           \int_0^\infty \e^{2\theta y-\frac43y^{3/2}}\d y\leq c_1\theta^2 \e^{\frac23\theta^3}.
          \end{align} 
\end{lemma}
\begin{proof}
          The estimate follows from a standard calculation in the steepest-descent method (see for instance \cite[Appendix B]{BDS16}). We first use change of variable to see that
          \begin{align*}
           \int_0^\infty \e^{2\theta y-\frac43y^{3/2}}\mathrm{d} y =\theta^2 \int_0^\infty \e^{\theta^3(2y-\frac43y^{3/2})}\mathrm{d} y.
          \end{align*} 
          Denote $h(y)= 2y-\frac43y^{3/2}$. It is clear that the function $h$ is increasing on $[0, \frac12]$, maximized at $1$ in $[\frac12, 9]$ and $h(y)\leq -2y$ for all $y\geq9$. 
          Hence,
          \begin{align*}
          \int_0^\infty \e^{\theta^3(2y-\frac34y^{3/2})}\mathrm{d} y &= \int_0^{\frac12}\e^{\theta^3(2y-\frac34y^{3/2})}\mathrm{d} y  + \int_{\frac12}^9\e^{\theta^3(2y-\frac34y^{3/2})}\mathrm{d} y
          + \int_9^{\infty}\e^{-2\theta^3y}\mathrm{d} y\\
          &\leq \frac12\e^{h(1/2)\theta^3} + \frac{17}{2}\e^{h(1)\theta^3}+ \frac{1}{2\theta^3}\e^{-18\theta^3}
          \leq c_1 \e^{h(1)\theta^3}= c_1\e^{\frac23\theta^3}.
          \end{align*}
          This proves Lemma \ref{lem:SD}.
\end{proof}

\begin{lemma}\label{lem:SD2}
          There exists a constant $c_2>0$ such that for all $\theta\geq1$,
          \begin{align}\label{eq:SD2}
           \int_{(-\infty, 0]^2}\d x \d y\, \e^{-\frac{(x+y)^2}{2} -2\theta(x+y)} \leq c_2\,\e^{4\theta^2}.
          \end{align}
\end{lemma}
\begin{proof}
         Using change of variable,
         \begin{align*}
         &\int_{(-\infty, 0]^2}\mathrm{d} x \mathrm{d} y\, \e^{-\frac{(x+y)^2}{2} -2\theta(x+y)}= \theta^2\int_{[0, \infty)^2} \mathrm{d} x \mathrm{d} y\, \e^{\theta^2(-\frac{(x+y)^2}{2} +2(x+y))}\\
         &\qquad\qquad\qquad \qquad\leq \theta^2\int_{[0,4]^2}\mathrm{d} x\mathrm{d} y\,  \e^{\theta^2(-\frac{(x+y)^2}{2} +2(x+y))} + 2\theta^2\int_0^4\mathrm{d} x\int_0^\infty \mathrm{d} y\,  \e^{\theta^2(-\frac{(x+y)^2}{2} +2(x+y))}\\
         &\qquad \qquad \qquad \qquad \quad + \theta^2 \int_{[4, \infty)^2}\mathrm{d} x \mathrm{d} y\, \e^{\theta^2(-\frac{(x+y)^2}{2} +2(x+y))}.
         \end{align*}
         It is clear that
         \begin{align*}
         \int_{[0,4]^2}\mathrm{d} x\mathrm{d} y\,  \e^{\theta^2(-\frac{(x+y)^2}{2} +2(x+y))} \leq 16 \e^{2\theta^2}. 
         \end{align*}
         Moreover, 
         \begin{align*}
         \int_0^4\mathrm{d} x\int_0^\infty \mathrm{d} y\,  \e^{\theta^2(-\frac{(x+y)^2}{2} +2(x+y))} \leq 4\int_{-\infty}^\infty \e^{\theta^2(-\frac{y^2}2 + 2y)}\mathrm{d} y = \frac{4}{\theta}\e^{2\theta^2}\int_{-\infty}^
         \infty \e^{-\frac12y^2}\mathrm{d} y. 
         \end{align*}
         Furthermore, 
         \begin{align*}
         \int_{[4, \infty)^2}\mathrm{d} x \mathrm{d} y\, \e^{\theta^2(-\frac{(x+y)^2}{2} +2(x+y))} \leq \left(\int_4^\infty \e^{\theta^2(-\frac{y^2}{2} +2y)} \mathrm{d} y  \right)^2 \leq  
         \theta^{-2}\e^{4\theta^2}\left(\int_{-\infty}^\infty \e^{-y^2/2} \mathrm{d} y  \right)^2.
         \end{align*}
         The preceding estimates prove Lemma \ref{lem:SD2}.
\end{proof}

\begin{lemma}\label{lem:SD3}
          There exists a constant $c_3>0$ such that for all $\theta\geq1$,
          \begin{align}\label{eq:SD3}
           \int_{(-\infty, 0]^2}\d x \d y\, y^2\e^{-\frac{(x+y)^2}{2} -2\theta(x+y)} \leq  c_3\,\theta^2\e^{4\theta^2}.
          \end{align}
\end{lemma}
\begin{proof}
         The proof is similar to that of Lemma A.2. Using change of variable,
         \begin{align*}
           &\int_{(-\infty, 0]^2}\mathrm{d} x \mathrm{d} y\, y^2\e^{-\frac{(x+y)^2}{2} -2\theta(x+y)} = \theta^4 \int_{[0, \infty)^2}\mathrm{d} x \mathrm{d} y\, y^2\e^{\theta^2(-\frac{(x+y)^2}{2} +2(x+y))} \\
           &\qquad \qquad \leq \theta^4 \int_{[0, 4]^2}\mathrm{d} x \mathrm{d} y\, y^2\e^{\theta^2(-\frac{(x+y)^2}{2} +2(x+y))}  + 
           2\theta^4  \int_0^4\mathrm{d} x \int_0^\infty\mathrm{d} y\, (x+y)^2\e^{\theta^2(-\frac{(x+y)^2}{2} +2(x+y))} \\
           &\qquad\qquad \quad + \theta^4 \int_{[4, \infty)^2}\mathrm{d} x \mathrm{d} y\, y^2\e^{\theta^2(-\frac{(x+y)^2}{2} +2(x+y))}.
          \end{align*}
          Evidently, 
          \begin{align*}
          \int_{[0, 4]^2}\mathrm{d} x \mathrm{d} y\, y^2\e^{\theta^2(-\frac{(x+y)^2}{2} +2(x+y))} \leq 256 \e^{2\theta^2}. 
          \end{align*}
          Moreover,
          \begin{align*}
          & \int_0^4\mathrm{d} x \int_0^\infty\mathrm{d} y\, (x+y)^2\e^{\theta^2(-\frac{(x+y)^2}{2} +2(x+y))}  \leq 4 \int_{-\infty}^\infty  y^2\e^{\theta^2(-\frac{y^2}{2} +2y)}\mathrm{d} y\\
          &\qquad\qquad\qquad \qquad\leq \frac{8}{\theta^3}\e^{2\theta^2} \int_{-\infty}^\infty  y^2\e^{-\frac{y^2}{2} }\mathrm{d} y
           + \frac{32}{\theta}\e^{2\theta^2} \int_{-\infty}^\infty  \e^{-\frac{y^2}{2} }\mathrm{d} y.
          \end{align*}
          Furthermore,
          \begin{align*}
          \int_{[4, \infty)^2}\mathrm{d} x \mathrm{d} y\, y^2\e^{\theta^2(-\frac{(x+y)^2}{2} +2(x+y))}&\leq \int_{-\infty}^\infty \e^{\theta^2(-\frac{x^2}{2} +2x)} \mathrm{d} x
          \int_{-\infty}^\infty y^2 \e^{\theta^2(-\frac{y^2}{2} +2y)} \mathrm{d} y
          \leq \frac{C}{\theta^2}\e^{4\theta^2}
          \end{align*}
          for some constant $C>0$. The preceding estimates prove Lemma \ref{lem:SD3}.
\end{proof}

\begin{lemma}\label{lem:SD4}
          There exists a constant $c_4>0$ such that for all $\theta\geq1$,
          \begin{align}\label{eq:SD4}
            \int_{\R^2\setminus(-\infty, 0]^2}\d x \d y\, \e^{-\frac{(x-y)^2}{2} -2\theta(x+y) } \leq  c_4\,\theta\e^{2\theta^2}.
          \end{align}
\end{lemma}
\begin{proof}
          Using change of variable,
         \begin{align*}
           \int_{\R^2\setminus(-\infty, 0]^2}\mathrm{d} x \mathrm{d} y\, \e^{-\frac{(x-y)^2}{2} -2\theta(x+y) }&=
            \theta^2\int_{\R^2\setminus(-\infty, 0]^2}\mathrm{d} x \mathrm{d} y\, \e^{\theta^2(-\frac{(x-y)^2}{2} -2(x+y)) }
            :=\theta^2(J_1+2J_2),
                     \end{align*}
         where
         \begin{align*}
         J_1&= \int_0^\infty\mathrm{d} x \int_0^\infty \mathrm{d} y\,  \e^{\theta^2(-\frac{(x-y)^2}{2} -2(x+y)) }, \quad
         J_2= \int_0^\infty\mathrm{d} x \int_{-\infty}^0 \mathrm{d} y\,  \e^{\theta^2(-\frac{(x-y)^2}{2} -2(x+y)) }.
                  \end{align*}
         We perform change of variables $x+y=u, x-y=v$ to see that
         \begin{align*}
         J_1 &= \frac12 \int_0^\infty \mathrm{d} u \int_{-u}^u\mathrm{d} v\,   \e^{\theta^2(-\frac{v^2}{2} -2u)} 
         \leq \int_0^\infty u\e^{-2\theta^2u}\mathrm{d} u = \frac1{\theta^4} \int_0^\infty u\e^{-2u}\mathrm{d} u. 
         \end{align*}
         Similarly, 
         \begin{align*}
         J_2&= \int_0^\infty\mathrm{d} x \int_0^{\infty} \mathrm{d} y\,  \e^{\theta^2(-\frac{(x+y)^2}{2} -2(x-y)) }
         = \frac12 \int_0^\infty \mathrm{d} u \int_{-u}^u\mathrm{d} v \, \e^{\theta^2(-\frac{u^2}{2} -2v) } \leq  \int_0^\infty u\e^{\theta^2(-\frac{u^2}{2} +2u) } \mathrm{d} u\\
         &\leq \e^{2\theta^2} \int_{-\infty}^\infty |u| \e^{\theta^2(u-2)^2/2}\mathrm{d} u \leq \e^{2\theta^2} \left(\int_{-\infty}^\infty |u| \e^{\theta^2u^2/2}\mathrm{d} u
         + 2\int_{-\infty}^\infty \e^{\theta^2u^2/2}\mathrm{d} u\right)\\
         &=\e^{2\theta^2} \left(\theta^{-2}\int_{-\infty}^\infty |u| \e^{u^2/2}\mathrm{d} u
         + 2\theta^{-1}\int_{-\infty}^\infty \e^{u^2/2}\mathrm{d} u\right).
         \end{align*}
           The bounds on $J_1, J_2$ prove Lemma \ref{lem:SD4}. 
\end{proof}

\begin{lemma}\label{lem:SD5}
          There exists a constant $c_5>0$ such that for all $\theta\geq1$,
          \begin{align}\label{eq:SD5}
            \int_{\R^2\setminus(-\infty, 0]^2}\d x \d y\, y^2\e^{-\frac{(x-y)^2}{2} -2\theta(x+y) } \leq  c_5\,\theta^3\e^{2\theta^2}.
          \end{align}
\end{lemma}
\begin{proof}
          The proof is similar to that of Lemma A.4. We write
          \begin{align*}
          \int_{\R^2\setminus(-\infty, 0]^2}\mathrm{d} x \mathrm{d} y\, y^2\e^{-\frac{(x-y)^2}{2} -2\theta(x+y)} 
          &= \theta^4 \int_{\R^2\setminus(-\infty, 0]^2}\mathrm{d} x \mathrm{d} y\, y^2\e^{\theta^2(-\frac{(x-y)^2}{2} -2(x+y))} \\
          &:=\theta^4(T_1+T_2+T_3),
          \end{align*}
          where
          \begin{align*}
          T_1&=\int_0^\infty\mathrm{d} x \int_0^\infty \mathrm{d} y\, y^2\e^{\theta^2(-\frac{(x-y)^2}{2} -2(x+y))}, \\
          T_2&=\int_0^{\infty}\mathrm{d} x \int^0_{-\infty} \mathrm{d} y\, y^2\e^{\theta^2(-\frac{(x-y)^2}{2} -2(x+y))}, \\
          T_3&=\int^0_{-\infty}\mathrm{d} x \int_0^{\infty} \mathrm{d} y\, y^2\e^{\theta^2(-\frac{(x-y)^2}{2} -2(x+y))}. 
          \end{align*}
          Using change of variables $x+y=u, x-y=v$,
          \begin{align*}
          T_1&=\frac12\int_0^\infty\mathrm{d} u \int_{-u}^u \mathrm{d} v\, \frac{(u-v)^2}{4}\e^{\theta^2(-\frac{v^2}{2} -2u)}\\
          &\leq \frac12 \int_0^\infty u^2\e^{-2\theta^2u}\mathrm{d} u \int_{-\infty}^\infty \e^{-\theta^2v^2/2} \mathrm{d} v 
          =\frac{1}{2\theta^5}\int_0^\infty u^2\e^{-2u}\mathrm{d} u \int_{-\infty}^\infty \e^{-v^2/2} \mathrm{d} v.
          \end{align*}
          Similarly,
          \begin{align*}
          T_2&=\int_0^{\infty}\mathrm{d} x \int_0^{\infty} \mathrm{d} y\, y^2\e^{\theta^2(-\frac{(x+y)^2}{2} -2(x-y))}
          = \frac12 \int_0^\infty \mathrm{d} u\int_{-u}^u \mathrm{d} v\, \frac{(u-v)^2}{4} \e^{\theta^2(-\frac{u^2}{2} -2v)} \\
          &\leq \frac12 \int_0^\infty \mathrm{d} u\int_{-u}^u\mathrm{d} v\, u^2 \e^{\theta^2(-\frac{u^2}{2} -2v)}
          \leq \int_0^\infty  u^3 \e^{\theta^2(-\frac{u^2}{2} +2u)}\mathrm{d} u \\
          &\leq \e^{2\theta^2} \int_{-\infty}^\infty |u+2|^3 \e^{-\theta^2u^2/2}\mathrm{d} u \leq 4\e^{2\theta^2} \int_{-\infty}^\infty 
          (|u|^3+8) \e^{-\theta^2u^2/2}\mathrm{d} u\\
          &= 4\e^{2\theta^2} \left(\theta^{-4}\int_{-\infty}^\infty 
          |u|^3\e^{-u^2/2}\mathrm{d} u +  8\theta^{-1}  \int_{-\infty}^\infty 
          \e^{-u^2/2}\mathrm{d} u\right),
          \end{align*}
          where in the fourth inequality, we used the inequality $(|a| +|b|)^3\leq 4(|a|^3+|b^3|)$.  Moreover,
          \begin{align*}
          T_3&=\int_0^{\infty}\mathrm{d} x \int_0^{\infty} \mathrm{d} y\, y^2\e^{\theta^2(-\frac{(x+y)^2}{2} +2(x-y))}= \frac12 \int_0^\infty \mathrm{d} u \int_{-u}^u\mathrm{d} v\, 
          \frac{(u-v)^2}{4} e^{\theta^2(-\frac{u^2}{2} +2v)}\\
          & \leq \frac12 \int_0^\infty \mathrm{d} u \int_{-u}^u\mathrm{d} v\, 
          u^2 e^{\theta^2(-\frac{u^2}{2} +2v)}\leq \int_0^\infty  u^3 \e^{\theta^2(-\frac{u^2}{2} +2u)}\mathrm{d} u \\
                    &\leq  4\e^{2\theta^2} \left(\theta^{-4}\int_{-\infty}^\infty 
          |u|^3\e^{-u^2/2}\mathrm{d} u +  8\theta^{-1}  \int_{-\infty}^\infty 
          \e^{-u^2/2}\mathrm{d} u\right),
          \end{align*}
          where the last inequality holds by the same arguments as in the estimate of $T_2$.  The bounds on $T_1, T_2, T_3$ prove Lemma \ref{lem:SD5}.          
\end{proof}

\end{appendix}

\noindent\textbf{Acknowledgement}.  Research supported in part by Beijing Natural Science Foundation (No. 1232010),
National Natural Science Foundation of China (No. 12201047) and 
National Key R\&D Program of China (No. 2022YFA 1006500). The author would like to thank Davar Khoshnevisan and Jeremy Quastel for useful discussions. The author is grateful to the editor and two anonymous referees for their valuable comments and suggestions which improved the presentation of this paper.

 \bigskip

 \begin{small}
\noindent\textbf{Fei Pu}
Laboratory of Mathematics and Complex Systems,
School of Mathematical Sciences, Beijing Normal University, 100875, Beijing, China.\\
Email: \texttt{fei.pu@bnu.edu.cn}\\
\end{small}


\begin{thebibliography}{999}

\bibitem{BBD08}
Baik, J., Buckingham, R. and DiFranco, J. (2008).
Asymptotics of Tracy-Widom distributions and the total integral of a Painlev\'e II function. 
{\it Comm. Math. Phys.} {\bf 280} 463--497. \MR{2395479} \url{https://doi.org/10.1007/s00220-008-0433-5}

\bibitem{BDS16}
Baik, J., Deift, P. and Suidan, T. (2016).
{\it Combinatorics and Random Matrix Theory}.
Graduate Studies in Mathematics, {\bf 172}. American Mathematical Society, Providence, RI.
\MR{3468920} 
\url{https://doi.org/10.1090/gsm/172}

\bibitem{BaZ23}
Balan, M. R. and Zheng, G. (2024).
Hyperbolic Anderson model with L\'{e}vy white noise: spatial ergodicity and fluctuation.
{\it Trans. Amer. Math. Soc.} {\bf 377} 4171--4221.
\MR{4748618}
\url{https://doi.org/10.1090/tran/9135}

\bibitem{BB24}
Basu, R. and Bhattacharjee, S. (2024).
Limit theorems for extrema of Airy processes. arXiv:2406.11826





\bibitem{BBF23}
Basu, R., Busani, O. and Ferrari, P.L. (2023).
On the exponent governing the correlation decay of the Airy$_1$ process. 
{\it Comm. Math. Phys.} {\bf 398} 1171--1211.
\MR{4561801}
\url{https://doi.org/10.1007/s00220-022-04544-1}

\bibitem{BoS02}
Borodin, A. N.  and Salminen, P. (2002).
{\it Handbook of Brownian Motion—Facts and Formulae}. 
Second edition. Probability and its Applications. Birkh\"auser Verlag, Basel, 2002
\MR{1912205}
\url{https://doi.org/10.1007/978-3-0348-8163-0}

\bibitem{Bul96}
Bulinski, A. V. (1996).
On the convergence rates in the CLT for positively and negatively dependent random fields. In {\it Probability Theory and Mathematical Statistics}, 3–14, Gordon and Breach, Amsterdam.
\MR{1661688}


\bibitem{BuS98}
Bulinskiĭ, A. V. and Shabanovich, \`E. (1998).
Asymptotic behavior of some functionals of positively and negatively dependent random fields. {\it Fundam. Prikl. Mat.}
{\bf 4} 479--492.
\MR{1801168}

\bibitem{CHH23}
Calvert, J., Hammond, A. and Hegde, M. (2023).
Brownian structure in the KPZ fixed point.
{\it Ast\'erisque} No. 441  v+119 pp.
\MR{4612529}
\url{https://doi.org/10.24033/ast.1200}

\bibitem{CKNP21}
Chen, L., Khoshnevisan, D., Nualart, D. and Pu, F. (2021).
Spatial ergodicity for SPDEs via Poincar\'{e}-type inequalities. 
{\it Electron. J. Probab.} {\bf 26}  1--37
\MR{4346664}
\url{https://doi.org/10.1214/21-ejp690}


\bibitem{CKNP23}
Chen, L., Khoshnevisan, D., Nualart, D., Pu, F. (2023).
Central limit theorems for spatial averages of the stochastic heat equation via Malliavin-Stein's method. 
{\it Stoch. Partial Differ. Equ. Anal. Comput. } {\bf 11} 122--176.
\MR{4563698}
\url{https://doi.org/10.1007/s40072-021-00224-8}

\bibitem{Che16}
Chen, X. (2016)
Spatial asymptotics for the parabolic Anderson models with generalized time-space Gaussian noise.
{\it Ann. Probab.} {\bf 44} 1535--1598.
\MR{3474477}
\url{https://doi.org/10.1214/15-AOP1006}

\bibitem{CJK13}
Conus, D., Joseph, M. nand Khoshnevisan, D. (2013).
On the chaotic character of the stochastic heat equation, before the onset of intermitttency.
{\it Ann. Probab.} {\bf 41} 2225--2260.
\MR{3098071}
\url{https://doi.org/10.1214/11-AOP717}

\bibitem{CoS14}
Corwin, I. and  Sun, X. (2014).
Ergodicity of the Airy line ensemble. 
{\it Electron. Commun. Probab.} {\bf 19}  1-11
\MR{3246968}
\url{https://doi.org/10.1214/ECP.v19-3504}

\bibitem{DuV13}
Dumaz, L. and Vir\'ag, B. (2013).
The right tail exponent of the Tracy-Widom $\beta$ distribution. 
{\it Ann. Inst. Henri Poincar Probab\'e. Stat.} {\bf 49} 915--933.
\MR{3127907}
\url{https://doi.org/10.1214/11-AIHP475}


\bibitem{Dur19}
Durrett, R. (2019).
{\it Probability: Theory and Examples}.
Fifth edition. Cambridge Series in Statistical and Probabilistic Mathematics, {\bf 49}. Cambridge University Press, Cambridge.
\MR{3930614}
\url{https://doi.org/10.1017/9781108591034}

\bibitem{EPW67}
	Esary, J. D., Proschan, F. and Walkup, D. W. (1967).
	Association of random variables with applications.
	{\it Ann.\ Math.\ Statist.}\ {\bf 38} 1466--1474. 
\MR{0217826}
\url{https://doi.org/10.1214/aoms/1177698701}

\bibitem{FeL24}
Ferrari, P. L. and Liu, M. (2024).
Exact decay of the persistence probability in the Airy$_1$ process. arXiv:2402.10661


\bibitem{Hag08}
H\"agg, J. (2008).
Local Gaussian fluctuations in the Airy and discrete PNG processes.
{\it Ann. Probab.} {\bf 36} 1059--1092
\MR{2408583}
\url{https://doi.org/10.1214/07-AOP353}

\bibitem{Leb72}
Lebowitz, J. L. (1972).
Bounds on the correlations and analyticity properties of ferromagnetic Ising spin systems.
{\it Comm. Math. Phys.} {\bf 28} 313--321.
\MR{0323271}
\url{http://projecteuclid.org/euclid.cmp/1103858445}

\bibitem{Leh66}
Lehmann, E. L. (1966).
Some concepts of dependence.
{\it Ann. Math. Statist.} {\bf 37} 1137--1153.
\MR{0202228}
\url{https://doi.org/10.1214/aoms/1177699260}

\bibitem{MQR21}
Matetski, K., Quastel, J. and Remenik, D. (2021).
The KPZ fixed point.
{\it Acta Math.} {\bf 227} 115--203.
\MR{4346267}
\url{https://doi.org/10.4310/acta.2021.v227.n1.a3}


\bibitem{New80}
Newman, C. M. (1980). Normal fluctuations and the FKG inequalities.
{\it Comm. Math. Phys.} {\bf 74} 119--128.
\MR{0576267}
\url{http://projecteuclid.org/euclid.cmp/1103907978}

\bibitem{New84}
Newman, C. M. (1984).
Asymptotic independence and limit theorems for positively and negatively dependent random variables. In  {\it Inequalities in Statistics and Probability (Lincoln, Neb., 1982, 127–140,
IMS Lecture Notes Monogr. Ser.}, {\bf 5}, Inst. Math. Statist., Hayward, CA.
\MR{0789244}
\url{https://doi.org/10.1214/lnms/1215465639}

\bibitem{Pim18}
Pimentel, L. P. R. (2018).
Local behaviour of airy processes. {\it J. Stat. Phys.} {\bf 173} 1614--1638.
\MR{3880997}
\url{https://doi.org/10.1007/s10955-018-2147-1}

\bibitem{PrS02}
Pr\"ahofer, M. and Spohn, H. (2002).
Scale invariance of the PNG droplet and the Airy process. 
{\it J. Statist. Phys.} {\bf 108} 1071--1106.
\MR{1933446}
\url{https://doi.org/10.1023/A:1019791415147}

\bibitem{PR12}
Prakasa Rao, B. L. S. (2012).
{\it Associated Sequences, Demimartingales and Nonparametric inference}.
Probability and its Applications. Birkh\"auser/Springer, Basel.
\MR{3025761}
\url{https://doi.org/10.1007/978-3-0348-0240-6}


\bibitem{QuR13}
Quastel, J. and Remenik, D. (2013).
Local behavior and hitting probabilities of the Airy$_1$ process.
{\it Probab. Theory Related Fields} {\bf 157} 605--634.
\MR{3129799}
\url{https://doi.org/10.1007/s00440-012-0466-8}

\bibitem{QuS23}
Quastel, J. and Sarkar, S. (2023).
Convergence of exclusion processes and the KPZ equation to the KPZ fixed point.
{\it J. Amer. Math. Soc.} {\bf 36} 251--289.
\MR{4495842}
\url{https://doi.org/10.1090/jams/999}


\bibitem{Sas05}
Sasamoto, T. (2005).
Spatial correlations of the 1D KPZ surface on a flat substrate.
{\it J. Phys. A} {\bf 38}  L549--L556.
\MR{2165697}
\url{https://doi.org/10.1088/0305-4470/38/33/L01}

\bibitem{TW94}
Tracy, C. A. and Widom, H. (1994).
Level-spacing distributions and the Airy kernel. 
{\it Comm. Math. Phys.} {\bf 159} 151--174.
\MR{1257246}
\url{http://projecteuclid.org/euclid.cmp/1104254495}

\bibitem{TW96}
Tracy, C. A. and Widom, H. (1996).
On orthogonal and symplectic matrix ensembles. 
{\it Comm. Math. Phys.} {\bf 177} 727--754.
\MR{1385083}
\url{http://projecteuclid.org/euclid.cmp/1104286442}

\bibitem{Vir20}
Vir\'ag, B. (2020). The heat and the landscape I. arXiv:2008.07241


\bibitem{WFS17}
Weiss, T., Ferrari, P. L. and Spohn, H. (2017).
{\it Reflected Brownian Motions in the KPZ Universality Class}.
SpringerBriefs in Mathematical Physics, {\bf18}. Springer, Cham.
\MR{3585775}
\url{https://doi.org/10.1007/978-3-319-49499-9}


\bibitem{Wid04}
Widom, H. (2004).
On asymptotics for the Airy process. 
{\it J. Statist. Phys.} {\bf 115} 1129--1134.
\MR{2054175}
\url{https://doi.org/10.1023/B:JOSS.0000022384.58696.61}

 \end{thebibliography}
\end{document}